\documentclass[a4paper, 12pt, twoside,openright]{article}

\usepackage{enumerate}
\usepackage{enumitem}
\usepackage[section]{placeins}	% Place un FloatBarrier à chaque nouvelle section
\usepackage{epigraph}
\usepackage[font={small}]{caption}
\usepackage{appendix}

\usepackage{amsmath,amssymb,mathrsfs}
\usepackage{amsfonts}			% Permet d'utiliser des polices mathématiques
\usepackage{nicefrac}			% Fractions 'inline'

\usepackage{bbm} % pour Id

\usepackage{stmaryrd}
\usepackage{calc,ifthen,xspace}

\usepackage[all,cmtip]{xy}
\usepackage{array}
\usepackage{multirow}
\usepackage{booktabs}
\usepackage{colortbl}
\usepackage{tabularx}
\usepackage{multirow}
\usepackage{threeparttable}

\usepackage{graphicx}			% Permet l'inclusion d'images
\usepackage{subcaption}
\usepackage{pdfpages}
\usepackage{rotating}
\usepackage{pgfplots}
	\usepgfplotslibrary{groupplots}
	
\usepackage{tikz}
	\usetikzlibrary{backgrounds,automata}
	\pgfplotsset{width=7cm,compat=1.3}
	\pgfplotsset{every linear axis/.append style={
		/pgf/number format/.cd,
		use comma,
		1000 sep={\,},
	}}
\usepackage{eso-pic}
\usepackage{import}

\usepackage{tikz-cd}
\usetikzlibrary{knots}

\usepackage{xspace}
\usepackage{textcomp}
\usepackage{array}
\usepackage{hyphenat}

\usepackage[pdftex,pdfborder={0 0 0},
			colorlinks=true,
			linkcolor=blue,
			citecolor=red,
			pagebackref=true,
			]{hyperref}	% Créera automatiquement les liens internes au PDF
\usepackage[top=2.5cm, bottom=2cm, left=3cm, right=2.5cm,
			headheight=15pt]{geometry}

\AtBeginDocument{}

\renewcommand{\leq}{\ensuremath{\leqslant}} %style des =<
\renewcommand{\geq}{\ensuremath{\geqslant}}

\DeclareMathOperator{\GL}{GL}
\DeclareMathOperator{\Z}{\mathbb Z}
\DeclareMathOperator{\F}{\mathbb F}
\DeclareMathOperator{\Lie}{\mathcal L}
\DeclareMathOperator{\Der}{Der}
\DeclareMathOperator{\ad}{ad}
\DeclareMathOperator{\tr}{Tr}
\DeclareMathOperator{\ima}{Im}
\DeclareMathOperator{\coker}{coker}

\DeclareMathOperator{\Hom}{Hom}
\DeclareMathOperator{\Aut}{Aut}

\DeclareMathOperator{\End}{End}

\DeclareMathOperator{\Act}{Act}

\DeclareMathOperator{\Ext}{Ext}
\DeclareMathOperator{\gr}{gr}

\usepackage{amsthm}

\makeatletter
\newtheorem*{rep@theorem}{\rep@title}
\newcommand{\newreptheorem}[2]{%
\newenvironment{rep#1}[1]{%
 \def\rep@title{#2 \ref{##1}}%
 \begin{rep@theorem}}%
 {\end{rep@theorem}}}
\makeatother

\theoremstyle{plain}
\newtheorem{theo}{Theorem}[section]
\newreptheorem{theo}{Theorem}
\newtheorem{lem}[theo]{Lemma}
\newtheorem{prop}[theo]{Proposition}
\newreptheorem{prop}{Proposition}

\newtheorem{cor}[theo]{Corollary}
\newreptheorem{cor}{Corollary}

\newcounter{PB}
\newtheorem{pb}[PB]{Problem}
\newreptheorem{pb}{Problem}

\numberwithin{equation}{theo}

\theoremstyle{definition}
\newtheorem{defi}[theo]{Definition}
\newtheorem{nota}[theo]{Notation}
\newtheorem{ex}[theo]{Example}
\newtheorem{rmq}[theo]{Remark}

\newcommand{\cat}[1]{\mathcal #1}
\renewcommand{\k}{\Bbbk}

\newcommand{\kMod}[1][\k]{\cat Mod _{#1}}
\newcommand{\kfAlg}[1][\k]{f\cat Alg _{#1}}
\newcommand{\Enstq}[2]{\left\{\ #1\ \middle|\ #2\ \right\}}    %£ ne pas couper ?
\newcommand{\Id}[1][]{\mathbbm 1_#1}                             %£ \mathds{1} dsfont ???
\newcommand{\cpx}[1]{#1_*}

\newcommand*{\longhookrightarrow}{\ensuremath{\lhook\joinrel\relbar\joinrel\rightarrow}}

\let\oldpagenumbering\pagenumbering
\renewcommand{\pagenumbering}[1]{%
	\cleardoublepage
	\oldpagenumbering{#1}
}
\author{Jacques \scshape{Darn\'e}}
\title{On the stable Andreadakis problem}
\date{\today}

\hypersetup{
    pdfauthor={Jacques Darn\'e},
    pdfsubject={Preprint},
    pdftitle={On the stable Andreadakis problem},
    pdfkeywords={algebraic topology, group theory, automorphisms of free groups}
}

\begin{document}

\maketitle

\begin{abstract}
Let $F_n$ be the free group on $n$ generators. Consider the group $IA_n$ of automorphisms of $F_n$ acting trivially on its abelianization. There are two canonical filtrations on $IA_n$: the first one is its lower central series $\Gamma_*$; the second one is the Andreadakis filtration $\mathcal A_*$, defined from the action on $F_n$. In this paper, we establish that the canonical morphism between the associated graded Lie rings $\Lie(\Gamma_*)$ and $\Lie(\mathcal A_*)$ is stably surjective. We then investigate a $p$-restricted version of the Andreadakis problem. A calculation of the Lie algebra of the classical congruence group is also included.
\end{abstract}

\section*{Introduction}
\addcontentsline{toc}{section}{Introduction}

Automorphisms of free groups have been widely studied over the years, from many different points of view. They are linked to the \emph{mapping class groups} of surfaces and braid groups \cite{Farb-Margalit}; they also act on a moduli space of graphs, called the \emph{outer space}, introduced in \cite{Culler-Vogtmann}, which is still actively studied nowadays (see, for instance, \cite{BBM}, or \cite{Francaviglia}). 
Recently, several results have also been obtained regarding the stable homology of these groups \cite{Galatius, RW-Wahl, Djament-Vespa, Djament}. 

One way to try and understand the structure of these automorphism groups is to cut them into pieces, by considering a family of subgroups and studying how these interact with each other. Such families of subgroups can arise from the action on the free group $F_n$ and related geometric objects, as is the case with the \emph{automorphisms with boundaries} (see for instance \cite{Jensen-Wahl} or \cite{Day-Putman_Birman_seq}), and for the \emph{Andreadakis subgroups}, which we now focus on.

The first Andreadakis subgroup of $\Aut(F_n)$ is the $IA$-group. Precisely, we can first look at how automorphisms act on $F_n^{ab} \cong \Z^n$. That is, we can consider the projection from $\Aut(F_n)$ onto $GL_n(\Z)$. We then put aside this linear part by considering only $IA_n$, the subgroup of automorphisms acting trivially on $\Z^n$, which is an algebraic analogue of the \emph{Torelli subgroup} of the mapping class group. An explicit finite set of generators of $IA_n$ has been known for a long time \cite{Nielsen} -- see also \cite[5.6]{BBM}. Nevertheless, the structure of $IA_n$ remains largely mysterious. For instance, $IA_3$ is not finitely presented \cite{Krstic}, and it is not known if $IA_n$ is finitely presented for $n > 3$. Recent results about the $IA$-groups include the finite $L$-presentation of $IA_n$ given in \cite{Day-Putman}, or finiteness results on the lower central series of $IA_n$ obtained in \cite{Church}.

The $IA$-group is the first step of the \emph{Andreadakis filtration} $IA_n = \mathcal A_1 \supseteq \mathcal A_2 \supseteq \cdots $, in which $\mathcal A_j$ is the group of automorphisms acting trivially on $F_n/\Gamma_{j+1}(F_n)$, where $F_n$ is filtered by its lower central series $F_n = \Gamma_1(F_n) \supseteq \Gamma_2(F_n) \supseteq\cdots$. The Andreadakis filtration is an \emph{$N$-series}. As such, it contains the minimal $N$-series on $IA_n$, its lower central series: for all $k$, $\mathcal A_k \supseteq \Gamma_k(IA_n)$. We are thus led to the problem of comparing these filtrations, that we call the \emph{Andreadakis problem}. 

Since the two filtrations are $N$-series, the associated graded objects are graded Lie rings (that is, Lie algebras over $\Z$), the Lie bracket being induced by the commutator map $(x, y) \mapsto [x,y] = xyx^{-1}y^{-1}$. The inclusion $i:\Gamma_*(IA_n) \subseteq \mathcal A_*$ induces a morphism of Lie rings:
\begin{equation}\label{Lie_morphism}
i_*: \Lie(\Gamma_*(IA_n)) = \bigoplus\limits_{j \geq 1} \Gamma_j(IA_n)/\Gamma_{j+1}(IA_n) \longrightarrow \Lie(\mathcal A_*) = \bigoplus\limits_{j \geq 1} \mathcal A_j/\mathcal A_{j+1}.
\end{equation}
Thus, the Andreadakis problem translates into the following question:
\begin{pb}[Andreadakis]\label{pb_Andreadakis}
How close is the morphism \eqref{Lie_morphism} to be an isomorphism? 
\end{pb}
Andreadakis conjectured that the filtrations were the same \cite[p.\ 253]{Andreadakis}. In \cite{Bartholdi}, Bartholdi disproved the conjecture, using computer calculations.  He then tried to prove that the two filtrations were the same up to finite index, but in the erratum \cite{Bartholdi1}, he showed that even this weaker statement cannot be true. His proof uses the $L$-presentation of $IA_n$ given in \cite{Day-Putman}, to which he applies algorithmic methods described in \cite{BBH} to calculate (using the software GAP) the first degrees of the Lie algebra associated to each filtration. 

\medskip

In this paper, we are interested in the difference between $\mathcal A_k(F_n)$ and $\Gamma_k(IA_n)$ for $n \gg k$, that is, in the stable range. We thus ask the following question:
\begin{pb}[Andreadakis - stable version]\label{pb_Andreadakis_stable}
How close is the morphism 
\begin{equation}\label{Lie_morphism_k}
i_*: \Lie_k(\Gamma_*(IA_n)) = \Gamma_k(IA_n)/\Gamma_{k+1}(IA_n) \longrightarrow \Lie_k(\mathcal A_*) = \mathcal A_k/\mathcal A_{k+1}
\end{equation} 
to be an isomorphism when $n \gg k$? 
\end{pb}

Our main goal here is to show the following partial answer to this question.
\begin{reptheo}{stable_surj}[Stable surjectivity]
When $n \geq k+2$, the morphism \eqref{Lie_morphism_k} is surjective.
\end{reptheo}

A (weaker) rational version of this theorem has been obtained independently by Massuyeau and Sakasai \cite[th. 5.1]{Massuyeau2}. Like them, we prove it by building on results from \cite{Satoh2}, but using quite different methods. These methods include a description of Andreadakis-like filtrations \emph{via} a categorical framework, allowing us to state and study a $p$-restricted version of the problem. We answer the questions asked in \cite{Massuyeau1} about this problem, and use our answers to study the stable $p$-restricted Andreadakis problem. Also, we solve the stable $q$-torsion Andreadakis problem for $\Z^n$, getting a complete calculation of the Lie ring of the congruence group $\GL_n(p \Z)$ for $n \geq 5$.

\bigskip

Let us now describe in more detail the methods we use and the results contained in the present paper. In \textbf{section \ref{Generalities}}, we set up a general framework for understanding $N$-series and their associated Lie algebras. We introduce a category $\mathcal{SCF}$ of $N$-series. We remark that the categorical definition of an \emph{action} of an object on another makes sense in this category. This allows us to interpret an old construction of Kaloujnine (see theorem \ref{Kaloujnine}) as the construction of universal actions in $\mathcal{SCF}$. This category is thus \emph{action-representative}, a situation studied in \cite{BJK, Borceux-Bourn, Bourn}. 
Using this language, we are able to recover and generalize several classical constructions:
\begin{itemize}
\item Taking the graded rings associated to $N$-series gives a functor $\Lie$ from $\mathcal{SCF}$ to the category of Lie rings. This functor preserves actions, and the \emph{Johnson morphism} admits a nice generalisation as the classifying morphism associated to an action between Lie rings obtained from an action in $\mathcal{SCF}$. 
\item \emph{Lazard's classical construction} of $N$-series from algebra filtrations described in \cite{Lazard} is recovered as a particular case of Kaloujnine's construction.
\item We also obtain the \emph{filtrations on congruence groups} studied in \cite{Lopez}.
\end{itemize}
In particular, we show that the filtration given by the last construction on the classical congruence group $\GL_n(q \Z)$ coincides with its lower central series when $n \geq 5$. As a consequence, we get an explicit calculation of this group's Lie ring (generalizing \cite[Th.\ 1.1]{Lee-Szczarba}, which is the degree-one part):
\begin{repcor}{Lie(GL_n(qZ))}
For all $n \geq 5$ and all $q \geq 3$, there is a canonical isomorphism of graded Lie rings (in degrees at least one):
\[\Lie(GL_n(q\Z)) \cong  \mathfrak{sl}_n(\Z/q)[t],\]
where the degree of $t$ is $1$, and the Lie bracket of $Mt^i$ and $Nt^j$ is $[M,N]t^{i+j}$.
\end{repcor}

\medskip

\textbf{Section \ref{Section_stable_surj}} deals with the proof of our stable surjectivity result (Theorem \ref{stable_surj}). The proof relies on the constructions of the first section, applied to Fox's free differential calculus. The \emph{Jacobian matrix} map $D : f \mapsto Df$ turns out to be a \emph{derivation} from $\Aut(F_n)$ to $\GL_n(\Z F_n)$, sending the Andreadakis filtration to the congruence filtration $\GL_n((I F_n)^*)$ (the group algebra $\Z F_n$ being filtered by the powers of its augmentation ideal $IF_n$). We then study such derivations, and the maps they induce on the graded Lie rings associated to $N$-series they preserve. We thus show that the \emph{trace map} defined by $\tr(f) = \tr(Df-\Id{n})$ induces a well-defined map:
\[\tr: \Lie(\mathcal A_*) \longrightarrow \gr(\Z F_n).\]
The graded algebra $\gr(\Z F_n)$ is in fact the tensor algebra $TV$ over $V = F_n^{ab}$. A result from \cite{Bryant} implies that this trace map takes values in $[TV, TV]$. Studying the links between free differential calculus and differential calculus in $TV$, we show that this trace map is exactly the one introduced by Morita \cite[Def. 6.4]{Morita}, getting the explicit description in terms of contraction maps notably used by Satoh in \cite{Satoh2}. Denoting the Johnson morphisms by $\tau$ and $\tau'$, we get a commutative diagram of graded linear maps:
\[\begin{tikzcd}
\Lie(IA_n) \ar[d, swap, "i_*"] \ar[rd, "\tau'"] &&\\
\Lie(\mathcal A_*) \ar[r, hook, "\tau"] 
&\Der(\mathfrak LV) \ar[r, "\tr_M"]
&TV,
\end{tikzcd}\]
where $\tr_M \circ \tau = \tr$
This gives the following inclusions of subspaces of $\Der(\mathfrak LV)$:
\[\ima(\tau') \subseteq \Lie(\mathcal A_*) \subseteq \tr_M^{-1}([TV,TV]).\]
We observe that calculations from \cite{Satoh2} work over $\Z$. From this , we deduce that the subspaces $\ima(\tau')$ and $\tr_M^{-1}([TV,TV])$ are stably the same, so these inclusions are equalities in the stable range, and $i_*$ must be stably surjective. We close the section by investigating some of the consequences of this result for automorphisms of free nilpotent groups.

\medskip

In \textbf{section \ref{Section-p}}, we turn to the \emph{$p$-restricted} version of the Andreadakis problem. Precisely, we can do the same construction as above, replacing the lower central series $\Gamma_*(F_n)$ by the mod-$p$ lower central series $\Gamma_*^{[p]}(F_n)$, which is an $N_p$-series:
\[(\Gamma_i^{[p]})^p \subseteq \Gamma_{pi}^{[p]}.\] 
Kaloujnine's construction gives an associated Andreadakis filtration $\mathcal A_*^{[p]}$ on the group $IA^{[p]}$ of automorphisms of $F_n$ acting trivially on $F_n^{ab} \otimes \F_p$. This filtration was shown in \cite{Massuyeau1} to be an $N_p$-series. It then contains the minimal $N_p$-series $\Gamma_*^{[p]}(IA^{[p]})$. Whence the natural question:
\begin{pb}[Andreadakis -- $p$-restricted version]\label{pb_Andreadakis-p}
What is the difference between the $N_p$-series $\mathcal A_*^{[p]}$ and $\Gamma_*^{[p]}(IA_n^{[p]})$ ? 
\end{pb}
Answering to \cite[rk. 8.6]{Massuyeau1}, we show that these two filtrations fit in the same kind of nice machinery as their classical counterparts, but they turn out to always differ. The paper ends on a quantification of the lack of stable surjectivity in this $p$-restricted case (see Proposition \ref{lack_of_stable_surj-p}).

\bigskip

\noindent
\textbf{Acknowledgements}: This work is part of the author's PhD thesis. The author is indebted to his advisors, Antoine Touz\'e and Aur\'elien Djament, for their constant support, countless helpful discussions, and numerous remarks and comments on earlier versions of the present paper. He also thanks Gwena\"el Massuyeau and Takao Satoh, who kindly agreed to be the reviewers of his thesis, for their useful observations and comments.

\tableofcontents

\section{Generalities on strongly central series}\label{Generalities}

\subsection{Notations and reminders}\label{rappels}

Throughout the paper, $G$ will denote an arbitrary group, and $\k$ a commutative unitary ring. The left and right action of $G$ on itself by conjugation are denoted respectively by $x^y = y^{-1}xy$ and ${}^y\! x = yxy^{-1}$.
The \emph{commutator} of two elements $x$ and $y$ in $G$ is denoted by:
\[[x,y]:= xyx^{-1}y^{-1}.\]
If $A$ and $B$ are subsets of $G$, we denote by $[A, B]$ the subgroup generated by the commutators $[a,b]$ with $(a,b) \in A \times B$.
If $A$ and $B$ are stable by conjugation by elements of $G$ (resp. by all automorphisms of $G$), then $[A,B]$ is a normal (resp. characteristic) subgroup of $G$.
For instance, $[G,G]$ is a characteristic subgroup of $G$, called the  \emph{derived subgroup} of $G$. The quotient $G^{ab}:= G/[G,G]$ is the \emph{abelianization} of $G$, its bigger abelian quotient. 
The derived subgroup is the second step of a filtration of $G$ by characteristic subgroups:

\begin{defi}
The \emph{lower central series} of $G$, denoted by $\Gamma_*(G)$, or shortly $\Gamma_*$, is the filtration of $G$ defined by:
\[\begin{cases} \Gamma_1:= G, \\ \Gamma_{k+1}:= [G, \Gamma_k]. \end{cases}\]
\end{defi}

\begin{defi}
A group $G$ is said to be \emph{nilpotent} if its lower central series stops. The least integer $c$ such that $\Gamma_{c+1}(G) = \{1\}$ is then $G$'s \emph{nilpotency class}.
More generally, $G$ is said to be \emph{residually nilpotent} if its lower central series is \emph{separated}, \emph{i.e.} if: 
\[\bigcap\limits_i \Gamma_i(G) = \{1\}.\] 
\end{defi}
One can easily check the following formulas:
\begin{prop}\label{formules}
For all $x, y, z \in G$,
\begin{itemize} [itemsep=-0.3em,topsep=0pt]
  \item $[x,x] = 1$,
  \item $[x,y]^{-1} = [y, x]$,
  \item $[x, yz] = [x,y] \left(^y[x,z] \right),$
\smallskip
	\item $\left[[x,y], {}^y\! z \right] \cdot \left[[y,z], {}^z\! x \right] \cdot \left[[z,x], {}^x\! y\right] = 1,$
\smallskip
	\item $\left[[x,y^{-1}], z^{-1} \right]^x \cdot \left[[z,x^{-1}], y^{-1} \right]^z \cdot \left[[y,z^{-1}], x^{-1}\right]^y = 1.$
\end{itemize}
\end{prop}

The last ones are two versions of the \emph{Witt-Hall identity}, which implies the following:

\begin{lem}[$3$-subgroups lemma]\label{lem3sgrp}
Let $A$, $B$ and $C$ be three subgroups of a group $G$. If two of the three following subgroups are trivial, then so is the third:
\[[A, [B, C]],\ \ \ [B, [C, A]],\ \ \ [C, [A, B]].\]
Equivalently, one of them is contained in the normal closure of the two others.
\end{lem}

\subsection{Strongly central filtrations and Lie algebras}

The theory of strongly central series has notably been studied by M. Lazard \cite{Lazard}.

\begin{defi}\label{def_SCF}
Let $G$ be a group. A \emph{strongly central filtration} of $G$ (also called \emph{strongly central series} or \emph{$N$-series}) is a filtration 
\[G = G_1 \supseteq \cdots \supseteq G_i \supseteq \cdots\] 
of $G$ by subgroups, satisfying:
\[\forall i, j \geq 1,\ [G_i, G_j] \subseteq G_{i+j}.\]
\end{defi}

Remark that indexation has to begin from $G = G_1$. In particular, $[G, N_i] \subseteq N_{i+1} \subseteq N_i$, which means exactly that the $N_i$ are \emph{normal} subgroups of $G$.

\medskip

\begin{prop}\label{gammaSCF}
Let $G$ be a group. The lower central series $\Gamma_*(G)$ is a strongly central series on $G$, and it is the minimal one.
\end{prop}

\begin{proof}
The strong centrality is shown by induction, using the $3$-subgroup lemma \ref{lem3sgrp}. Given a strongly central filtration $G_*$ of $G = G_1$, a straightforward induction then gives: $G_i \supseteq \Gamma_i(G)$ for any $i \geq 1$.
\end{proof}

\medskip

Let $G = G_1 \supseteq \cdots \supseteq G_k \supseteq \cdots$ be any strongly central filtration of a group $G$.
The quotients $\Lie_i(N_*):= N_i/N_{i+1}$ are abelian (for any $i \geq 1$), since $[N_i, N_i] \subseteq N_{2i} \subseteq N_{i+1}.$
The graded abelian group
\[\Lie(N_*):= \bigoplus\limits_{i \geq 1}{\Lie_i(N_*)},\]
is endowed with a bracket induced by the commutator map $(x, y) \mapsto [x, y]$ of $G$. Using the formulas \ref{formules}, one easily checks that this defines a Lie bracket: $\Lie(G_*)$ is a Lie ring (\emph{i.e.\ }a Lie algebra over $\Z$).

\medskip

\begin{nota}\label{nota_Lie}
We denote $\Lie(\Gamma_*(G))$ by $\Lie(G)$ (that is, if we do not specify a strongly central filtration on a group, it is understood to be filtered by its lower central series).
\end{nota}

\begin{ex}\label{L(Fn)_free}
If $G $ is a free group, then $\Lie(G)$ is the free Lie algebra over the $\Z$-module $G^{ab}$ \cite[th.\ 4.2]{Lazard}.
\end{ex}

As products of commutators become sums of brackets inside the Lie algebra, the following fundamental property follows from the definition of the lower central series:
\begin{prop}\label{engdeg1} 
The Lie ring $\Lie(G)$ is \emph{generated in degree $1$}. Precisely, it is generated (as a Lie ring) by $\Lie_1(G) = G^{ab}$. As a consequence, if $G$ is of finite type, then each $\Lie_n(G)$ is too.
\end{prop}

\subsection{Actions in the category of strongly central filtrations}\label{Actions}

Let $\mathcal{SCF}$ be the category whose objects are the strongly central filtrations, where morphisms between $G_*$ and $H_*$ are the group morphisms from $G_1$ to $H_1$ preserving filtrations. There is a forgetful functor $\omega_1: \mathcal{SCF} \longrightarrow \mathcal Grps$ defined by  $G_* \mapsto G_1$. This functor admits a left adjoint $\Gamma: G \mapsto \Gamma_*(G)$ (see Proposition \ref{gammaSCF}). It also admits a right adjoint $G \mapsto (G, G, ...)$.

\begin{prop}\label{SCF_cocomplete}
The category $\mathcal{SCF}$ is complete and cocomplete, and is homological (but not semi-abelian).
\end{prop}

A general reference on homological categories is \cite{BB}. The reader can also consult \cite{Hartl} for a simple version of the axioms defining homological cocomplete and semi-abelian categories.

\begin{proof}[Proof of Proposition \ref{SCF_cocomplete}]
The forgetful functor $\omega_1$ admits both a left and a right adjoint, so it has to commute to limits and colimits. It does not create either of them (in the sense of \cite{McLane}, V.1), but it almost does.

Precisely, let $F: \mathcal D \longrightarrow \mathcal{SCF}$ be a diagram. The colimit $G_\infty$ of the group diagram $\omega_1 F$ is in general endowed with several strongly central filtrations making $\omega_1 F(d) \longrightarrow G_\infty$ into filtration-preserving morphisms (for instance the trivial one). One checks easily that the minimal such filtration (which is the intersection of all those) is the colimit of $F$.

Similarly, the limit $G^\infty$ of the group diagram $\omega_1 F$ is endowed with several strongly central filtrations making $\varphi_d: G^\infty \longrightarrow \omega_1 F(d)$ into filtration-preserving morphisms (for instance its lower central series). However, the maximal such filtration is the limit of $F$. It is explicitly described as:
\[G^\infty_* = \bigcap\limits_d \varphi_d^{-1}(F(d)).\]  

To check that $\mathcal{SCF}$ is homological, one can check the axioms given in \cite{Hartl}. It is not semi-abelian, because there are equivalence relation $R_* \subseteq G_*^2$ for which $R_*$ is not the induced filtration on $R_1$ (like in topological groups -- or more generally in the categories of topological algebras considered in \cite{Borceux-Clementino} -- where an equivalence relation does not have to be endowed with the induced topology).
\end{proof}

In a homological category, we need to distinguish between usual epimorphisms (resp. monomorphisms) and \emph{regular} ones, that is, the ones obtained as coequalizers (resp. equalizers). In $\mathcal{SCF}$, the former are the $u$ such that $u_1 = \omega_1(u)$ is an epimorphism (resp. a monomorphism), whereas the latter are \emph{surjections} (resp. \emph{injections}):
\begin{defi}\label{def_inj-surj}
Let $u: G_* \longrightarrow H_*$ be a morphism in $\mathcal{SCF}$. It is called an \emph{injection} (resp.\ a \emph{surjection}) when $u_1$ is injective (resp.\ surjective) and $u^{-1}(H_i) = G_i$ (resp.\ $u(G_i) = H_i$) for all $i$.
\end{defi}

Examples of homological categories include abelian categories, the category $\mathcal Grps$ of groups, or the category $\mathcal Lie$ of Lie algebras. The usual lemmas of homological algebra (the nine lemma, the snake lemma, the five lemma...) are true in these categories. Homological categories differ from abelian ones notably by the fact that in general, two split extensions between the same objects are not isomorphic (by an isomorphism preserving the splittings). This allows us to define an action of an object on another. 

\begin{defi}\label{def_action}
Let $\mathcal C$ be a homological category. If $X$ and $Z$ are two objects of $\mathcal C$, we define an \emph{action} of $Z$ on $X$ as a split extension (with a given splitting):
\[\begin{tikzcd} X \ar[r] & Y \ar[r] & Z. \ar[l, bend right]
\end{tikzcd}\]
When such an action is given, we will say that $Z$ \emph{acts on} $X$, and write: $Z \circlearrowright X$.
\end{defi}
This definition (which needs only the weaker setting of \emph{pointed protomodular categories} to make sense) is motivated by the situation in $\mathcal Grps$, where an action of a group $K$ on a group $G$ is encoded by a semi-direct product structure $G \rtimes K$.

\begin{rmq} The choice of splitting is crucial here. For instance, the canonical extension:
\[\begin{tikzcd}[ampersand replacement=\&] X \ar[r, "\scalebox{0.7}{$\begin{pmatrix} 1 \\ 0 \end{pmatrix}$}"] \& X \times X \ar[r, "\scalebox{0.7}{$\begin{pmatrix} 0 & 1 \end{pmatrix}$}"] \& X
\end{tikzcd}\]
can be split by \scalebox{0.6}{$\begin{pmatrix} 0 \\ 1 \end{pmatrix}$}, or by the diagonal \scalebox{0.6}{$\begin{pmatrix} 1 \\ 1 \end{pmatrix}$}. The first choice gives the trivial action, whereas the second one gives the \emph{adjoint action}, which is highly non-trivial: in $\mathcal Lie$, this gives the adjoint representation; in $\mathcal Grps$, we get the action of a group on itself by conjugation. 
\end{rmq}

\medskip

The set $\Act(Z,X)$ of actions of $Z$ on $X$ is a contravariant functor in $Z$: the restriction of an action along a morphism is defined \emph{via} a pullback. In $\mathcal Grps$, as in $\mathcal Lie$, this functor is representable, for any $X$. Indeed, an action of a group $K$ on a group $G$ is given by a morphism $K \longrightarrow \Aut(G)$. Similarly, an action of a Lie algebra $\mathfrak k$ on a Lie algebra $\mathfrak g$ is given by a morphism $\mathfrak k \longrightarrow \Der(\mathfrak g)$, where $\Der(\mathfrak g)$ is the Lie algebra of \emph{derivations} from $\mathfrak g$ to itself. Recall that a derivation $\partial$ is a linear map satisfying:
\[\partial([x,y]) = [\partial x, y] + [x, \partial y].\]  
The situation when actions are representable has notably been studied in \cite{BJK}, and in several subsequent papers \cite{Borceux-Bourn, Bourn}. The following terminology was introduced in \cite[Def. 1.1]{Borceux-Bourn}:
\begin{defi}
A homological category $\mathcal C$ is said to be \emph{action-representative} when the functor $\Act(-,X)$ is representable, for any object $X \in \mathcal C$.
\end{defi}

Our goal for the rest of this section is to construct universal actions in $\mathcal{SCF}$, getting in particular the following result:
\begin{prop}\label{SCF_action-rep}
The category $\mathcal{SCF}$ is action-representative.
\end{prop}

A representative for $\Act(-,X)$ is a universal action on $X$. Explicitly, it is an action of an object $\mathcal A(X)$ on $X$ such that any action $Z \circlearrowright X$ is obtained by restriction along a unique morphism $Z \rightarrow \mathcal A(X)$. For instance, in $\mathcal Grps$, the universal action on $G$ is:
\[\begin{tikzcd} G \ar[r, hook] & G \rtimes \Aut(G) \ar[r, two heads] & \Aut(G), \ar[l, bend right]
\end{tikzcd}\]
where the group $G \rtimes \Aut(G)$ is the \emph{holomorph} of $G$. Its underlying set is $G \times \Aut(G)$, endowed with the product defined by $(g, \sigma) \cdot (h, \tau):= (g \sigma(h), \sigma \tau).$

\medskip

The construction of universal actions in $\mathcal{SCF}$ is given by a theorem of Kaloujnine \cite{Kaloujnine1, Kaloujnine2}, quoted by Lazard in \cite[p. 117]{Lazard}:

\begin{theo}[Kaloujnine]\label{Kaloujnine}
Let $G_*$ be a strongly central series. Let $j \geq 1$ be an integer. Let $\mathcal A_j(G_*) \subseteq \Aut(G_*)$ be the subgroup of automorphisms acting trivially on every quotient $G_i/G_{i+j}$. Then $\mathcal A_*(G_*)$ is a strongly central series.
\end{theo}

\begin{nota}
As for Lie rings (Notation \ref{nota_Lie}) we abbreviate $\mathcal A_*(\Gamma_*(G))$ to $\mathcal A_*(G)$.
\end{nota}

We can rewrite the definition of $\mathcal A_j(G_*)$ given in the theorem as:
\[\mathcal A_j(G_*) = \ker\left(\Aut(G_*) \longrightarrow \prod\limits_i \Aut\left(G_i/G_{i+j}\right)\right).\]
Identifying $G$ and $\Aut(G)$ to the subgroups $G \times 1$ and $1 \times \Aut(G)$ of the holomorph $G \rtimes \Aut(G)$, we can define the commutator of an automorphism with an element of $G$:
\[[\sigma, g] = \sigma(g)g^{-1}.\]
Note that $[\Aut(G), G] \subseteq G$.
Using this point of view, we can rephrase the previous definition:
\begin{equation}\label{def_A_*}
\mathcal A_j(G_*) = \Enstq{\sigma \in \Aut(G_*)}{\forall i \geq 1,\ [\sigma, G_i] \subseteq G_{i+j}}\ \subseteq\ G_1 \rtimes \Aut(G_*).
\end{equation}

\smallskip

\begin{proof}[Proof of Theorem \ref{Kaloujnine}]
We abbreviate $\mathcal A_j(G_*)$ to $\mathcal A_j$. Obviously, $\mathcal A_{j+1} \subseteq \mathcal A_j$. 
We show the strong centrality using the 3-subgroup lemma (Lemma \ref{lem3sgrp}).
Precisely, let $\alpha, \beta \geq 1$ be two integers. For all $i \geq 1$, the group $G_{i+\alpha+\beta}$ is normal in $G \rtimes \Aut(G_*)$ (it is normal in $G$ and $\Aut(G_*)$-stable). Lemma \ref{lem3sgrp} thus implies:
\[[[\mathcal A_\alpha,\mathcal A_\beta], G_i]\subseteq G_{i+\alpha+\beta}.\]
This says exactly that $[\mathcal A_\alpha,\mathcal A_\beta] \subseteq \mathcal A_{\alpha + \beta},$
which is the desired conclusion.
\end{proof}

\begin{rmq}
The group $\mathcal A_1(G_*)$ is the group of automorphisms of $G_1$ preserving $G_*$ and acting trivially on $\Lie(G_*)$.
\end{rmq}

\begin{ex}\label{déf_IA}
Let $G_* = \Gamma_*(G)$. Then $\Lie(G_*)$ is generated in degree one as a Lie algebra. As a consequence $\mathcal A_1(G_*)$ is the subgroup of automorphisms acting trivially on the abelianization $G^{ab} = \Lie_1(G)$, denoted by $IA_G$.
\end{ex}

In order to show that the filtration $\mathcal A_*(G_*)$ acts universally on $G_*$, we need to investigate actions in $\mathcal{SCF}$. 

\begin{prop}\label{actions_dans_SFC}
An action $K_* \circlearrowright G_*$ in $\mathcal{SFC}$ consists of a group action of $K = K_1$ on $G = G_1$ such that :
\[\forall i, j,\ [K_i, G_j] \subseteq G_{i+j}.\]
\end{prop}

\begin{proof}
Let an action of $K_*$ on $G_*$ be given:
\[\begin{tikzcd} G_* \ar[r, hook] & H_* \ar[r, two heads] & K_*. \ar[l, bend right]
\end{tikzcd}\]
The forgetful functor $\omega_1: G_* \mapsto G_1$ from $\mathcal{SCF}$ to groups has a left adjoint $G \mapsto \Gamma_*G$. Hence, it commutes with kernels. Applying this functor to the given action, we get a split extension of groups:
\[\begin{tikzcd} G_1 \ar[r, hook] & H_1 \ar[r, two heads] & K_1. \ar[l, bend right]
\end{tikzcd}\]
The group $H_1$ thus decomposes as a semi-direct product $G_1 \rtimes K_1$ of $G_1$ by $K_1$. 

In fact, there are other forgetful functors $\omega_i: G_* \mapsto G_i$, each one with a left adjoint $G \mapsto \Gamma_{\lceil \frac{*}{i}\rceil}G$ (where $\lceil - \rceil$ is the usual ceiling function). We then get split extensions of groups:
\[\begin{tikzcd} G_i \ar[r, hook] & H_i \ar[r, two heads] & K_i. \ar[l, bend right]
\end{tikzcd}\]
The groups $H_i$ thus decomposes as semi-direct products $G_i \rtimes K_i$ of $G_i$ by $K_i$. 

As $H_*$ is a strongly central filtration, we can apply Lemma \ref{sdprod} below to get the desired relation. Conversely, let a group action be given as in the statement of the proposition. Using the same lemma, we see that $H_* = G_* \rtimes K_*$ is a strongly central filtration on $H = K \rtimes G$, and the corresponding split sequence in $\mathcal{SFC}$ is exact.
\end{proof}

\begin{lem}\label{sdprod}
Let $K \circlearrowright G$ be an action in $\mathcal Grps$, encoded in a semi-direct product structure $H = G \rtimes K$.
Let $G_*$ and $K_*$ be given filtrations on $G = G_1$ and $K = K_1$ respectively. Then the $H_i:= G_i \rtimes K_i$ are subgroups of $H$ defining a strongly central filtration of $H$ if and only if:
\[\left\{
\begin{array}{l}
K_* \text{ is a strongly central series on } K, \\
G_* \text{ is a strongly central series on } G, \\
\forall i,j,\ [K_j,G_i] \subseteq G_{i+j}.
\end{array}
\right.\]
\end{lem}

\begin{proof}
Suppose first that $(G_i \rtimes K_i)_i$ is strongly central. Then its projection $K_*$ on $K$ also is. So is its intersection $G_*$ with $G$. Hence, the conclusion follows from:
\[[K_j,G_i] \subseteq (G_{i+j} \rtimes K_{i+j}) \cap G = G_{i+j}.\]

Conversely, under the hypothesis listed above, $G_i$ is stable under the action of $K_i$, so the $H_i = G_i \rtimes K_i$ are subgroups of $H$. We then use the formulas \ref{formules} to compute $[kg, k'g']$ with $k$, $g$, $k'$, $g'$ in $K_i$, $G_i$, $K_j$ and $G_j$ respectively:
\begin{align*}
[kg, k'g'] &= [kg, k'] \cdot ^{g'}\! [kg, g'] \\
&=\ ^k\![g,k'] \cdot [k,k'] \cdot\ ^{k'k}\! [g,g'] \cdot\ ^{k'}[k,g'] \\
&\in G_{i+j} \cdot K_{i+j} \cdot G_{i+j} \cdot G_{i+j} = G_{i+j} \rtimes K_{i+j}.
\end{align*}
This completes the proof.
\end{proof}

We are now ready to show the result announced in Proposition \ref{SCF_action-rep}:

\begin{prop}\label{Universal_action}
Let $G_*$ be a strongly central series. The strongly central series $\mathcal A_*(G_*)$ acts canonically  on $G_*$, and this action is universal.
\end{prop}

\begin{proof}
That $\mathcal A_*(G_*)$ acts on $G_*$ follows from the formula \eqref{def_A_*}, Theorem \ref{Kaloujnine} and Proposition \ref{actions_dans_SFC}. 

Given an action of a strongly central series $K_*$ on $G_*$, the underlying group action is described by a unique morphism from $K_1$ to $\Aut(G_1)$. From Proposition \ref{actions_dans_SFC}, we deduce that this morphism sends $K_j$ into $\mathcal A_j(G_*)$. Conversely, any morphism from $K_*$ to $\mathcal A_*(G_*)$ in $\mathcal{SCF}$ gives a group action lifting to an action in $\mathcal{SCF}$ by Proposition \ref{actions_dans_SFC}.
\end{proof}

\begin{rmq}\label{A_on_K}
If a group $K$ acts on a group $G$, and $G_*$ is a strongly central filtration on $G = G_1$, we can pull back the canonical filtration $\mathcal A_*(G_*)$ by the associated morphism:
\[K \longrightarrow \Aut(G).\]
This gives a strongly central filtration $\mathcal A_*(K,G_*)$, maximal amongst strongly central filtrations on subgroups of $K$ which act on $G_*$ \emph{via} the given action $K \circlearrowright G$. It can be described explicitly as:
\[\mathcal A_j(K,G_*) = \Enstq{k \in K}{\forall i \geq 1,\ [k, G_i] \subseteq G_{i+j}}\ \subseteq K.\]
\end{rmq}

\subsection{Johnson's morphisms}\label{Johnson_section}

The construction of the Johnson morphism associated with an action in $\mathcal{SCF}$ relies on the following:
\begin{prop}\label{exactness_of_L}
The Lie functor $\Lie: \mathcal{SCF} \longrightarrow \mathcal Lie_{\Z}$ is \emph{exact}, \emph{i.e.\ }it preserves short exact sequences \cite[Def 4.1.5]{BB}.
\end{prop} 
\begin{proof}
The exactness of $\Lie$ is equivalent to the exactness of each $\Lie_i: \mathcal{SCF} \longrightarrow Ab$. Consider the forgetful functors $\omega_i: G_* \mapsto G_i$ from $\mathcal{SCF}$ to $\mathcal Grps$. Each $\omega_i$ has a left adjoint $G \mapsto \Gamma_{\lceil \frac{*}{i}\rceil}G$, so these functors preserve kernels. Moreover, they preserve regular epimorphisms, which are surjections in $\mathcal{SCF}$ (cf. Definition \ref{def_inj-surj}). Hence, they are exact. Since $\Lie_i$ is the cokernel of the injection $\omega_{i+1} \hookrightarrow \omega_i$, its exactness follows from the nine lemma in $\mathcal Grpes$. Precisely, if $\mathcal E$ is any short exact sequence in $\mathcal{SCF}$, apply the nine lemma to the diagram
$\omega_{i+1}(\mathcal E) \hookrightarrow \omega_i(\mathcal E) \twoheadrightarrow \Lie_i(\mathcal E)$
to get the exactness of $\Lie_i(\mathcal E)$ from that of $\omega_i(\mathcal E)$ and $\omega_{i+1}(\mathcal E)$.
\end{proof}

As a consequence of Proposition \ref{exactness_of_L}, the functor $\Lie$ preserves actions. Precisely, from an action in $\mathcal{SCF}$:
\[\begin{tikzcd} G_* \ar[r, hook] & H_* \ar[r, two heads] & K_*, \ar[l, bend right]
\end{tikzcd}\]
we get an action in the category of graded Lie rings:
\[\begin{tikzcd} \Lie(G_*) \ar[r, hook] & \Lie(H_*) \ar[r, two heads] & \Lie(K_*). \ar[l, bend right]
\end{tikzcd}\]
Such an action is given by a morphism of graded Lie rings:
\begin{equation}\label{Johnson}
\Lie(K_*) \longrightarrow \Der_*(\Lie(G_*)).
\end{equation}
The target is the (graded) Lie algebra of graded derivations: a derivation is of degree $k$ when it raises degrees of homogeneous elements by $k$.
\begin{defi}
The morphism \eqref{Johnson} is called the \emph{Johnson morphism} associated to the given action $K_* \circlearrowright G_*$.
\end{defi}
We can give an explicit description of this morphism: for $k \in K$, the derivation associated to $\bar{k}$ is induced by $[\bar{k},-]$ inside $\Lie(G_* \rtimes K_*) = \Lie(G_*) \rtimes \Lie(K_*)$, so it is induced by $[k,-]$ inside $G_* \rtimes K_*$.

\begin{ex}\label{déf_tau}
The Johnson morphism associated to the universal action $\mathcal A_*(G_*)\circlearrowright G_*$ is the Lie morphism:
\[\tau: \Lie(\mathcal A_*(G_*)) \longrightarrow \Der_*(\Lie(G_*))\]
induced by $\sigma \mapsto (x \mapsto \sigma(x)x^{-1})$.
\end{ex}

If $K_1$ is in fact a normal subgroup of a group $K_0$, such that each $K_i$ is normal in $K_0$, and such that the action of $K_1$ on $G_*$ can be extended to an action of $K_0$, then $K_0$ acts on $\Lie (G_* \rtimes K_*)$ and this action factorizes through $K_0/K_1$. Moreover, as this action is by automorphisms of the Lie ring:
\begin{lem}\label{Johnson_equiv}
Let $K_0 \triangleright K_1$ be given as above. Then the Johnson morphism $\tau$ is $K_0/K_1$-equivariant.
\end{lem}

The action of $K_0$ on derivations is by conjugation. Precisely:
\[\tau(k\cdot x) = [k\cdot x,-] = k\cdot [x,k^{-1}\cdot -] = k \circ \tau(x) \circ k^{-1}.\]

\begin{lem}\label{Johnson_inj}
Let $K_* \circlearrowright H_*$ be an action in $\mathcal{SCF}$. The associated Johnson morphism $\tau: \Lie (K_*) \longrightarrow \Der(\Lie (G_*))$ is injective if and only if $K_*$ is the filtration $\mathcal A_*(K, G_*)$ defined in Remark \ref{A_on_K}.
\end{lem}

\begin{proof}
Every non-trivial element in $\ker(\tau_j)$ lifts to an element in $\mathcal A_j(K, G_*) - K_j$.
Conversely, an element $\sigma$ in $\mathcal A_j(K, G_*) - K_j$ is in $K_k - K_{k+1}$ for some $k < j$. Then $\overline{\sigma} \in  K_k /K_{k+1} - \{0\}$ is a non-trivial element in $\ker(\tau)$.
\end{proof}

\begin{rmq}
The definition of $\mathcal A_j(K, G_*)$ makes sense for $j=0$, giving a subgroup $K_0 = \mathcal A_0(K, G_*)$ of $K$ acting as above. The morphism $\tau$ is then $\mathcal A_0(K, G_*)$-equivariant. In fact, $\tau$ can then be extended to a morphism of \emph{extended Lie algebras}, in the sense of \cite{Massuyeau1}. Ideed, their construction of an algebra of \emph{extended derivations} is exactly a construction of universal actions in the category of extended Lie algebras, and their version of the Johnson morphism is exactly the one we find if we replace $N$-series and Lie algebras by their extended version in the constructions above.
\end{rmq}

\subsection{The Andreadakis problem}

Let $G$ be a group. To study the structure of $\Aut(G)$, we can consider first how automorphisms act on $G^{ab}$. Then we can put aside this linear part by considering  the kernel of the projection from $\Aut(G)$ to $GL(G^{ab})$. This kernel $IA_G$ is (residually) nilpotent when $G$ is, and in endowed with two strongly central filtrations: its lower central series, and the \emph{Andreadakis filtration} $\mathcal A_*(G)$. We are thus led to the problem of comparing these filtrations, which we call the \emph{Andreadakis problem} (Problem \ref{pb_Andreadakis}).

\medskip

Recall from Theorem \ref{Kaloujnine} and Example \ref{déf_IA} that $\mathcal A_*(G)$ is a strongly central filtration on
$\mathcal A_1(G) = IA_G$.
Since $\mathcal A_1(G)$ is the kernel of the canonical action of $\Aut(G)$ on $\Lie (G)$, we get an induced faithful action of $\Aut(G)/IA_G$ ($\leq GL(G^{ab})$) on $\Lie (G)$.
The next lemma gives a similar concrete description of all the $\mathcal A_j(G)$: 
\begin{lem}\label{Explicit_A_*}
The group $\mathcal A_j(G)$ is the subgroup of automorphisms acting trivially on $G/\Gamma_{j+1}(G)$, \emph{i.e.}
$\mathcal A_j(G) = \Enstq{\sigma \in \Aut(G)}{[\sigma, G] \subseteq \Gamma_{j+1}(G)}.$
\end{lem}

\begin{proof}
Let us denote by $K_j$ the right-hand side of the equality. We only need to show the inclusion $K_j \subseteq \mathcal A_j(G)$. The case $j=1$ is given in Example \ref{déf_IA}. Suppose it true for $j-1$, and let $\sigma \in K_j$. In particular, $\sigma$ is inside $K_{j-1} = \mathcal A_{j-1}(G)$, so it induces a derivation $[\sigma, -]$ of degree $j-1$ of $\Lie(G)$, \emph{via} the Johnson morphism. By definition, $[\sigma, G] \subseteq \Gamma_{j+1}$. This says exactly that the derivation $[\sigma, -]$ is trivial on $\Lie_1(G)$. Since $\Lie(G)$ is generated in degree one (Proposition \ref{engdeg1}), $[\sigma, -]$ has to be trivial on all of $\Lie(G)$: for all $i$, $[\sigma, \Gamma_i] \subseteq \Gamma_{i+j}$, which is the desired conclusion.
\end{proof}

The filtration $\mathcal A_*(G)$ is strongly central on $\mathcal A_1(G) = IA_G$. As a consequence, $\mathcal A_k(G)$ contains $\Gamma_k(IA_G)$. 
A first consequence of this inclusion is that the (residual) nilpotency of $G$ implies the (residual) nilpotency of $IA_G$. Precisely, let $G$ be a $c$-nilpotent group. Then $\mathcal A_c(G) = \{\Id{G}\}$. Accordingly, $\Gamma_c(IA_G) = \{\Id{G}\}$, so $IA_G$ is $(c-1)$-nilpotent. In a similar fashion, one can check that $IA_G$ has to be residually nilpotent when $G$ is.

The following question is crucial for trying to understand the structure of automorphism groups of residually nilpotent groups, in particular for trying to understand the structure of $\Aut(F_n)$:

\begin{reppb}{pb_Andreadakis}[Andreadakis]
What is the difference between $\mathcal A_*(G)$ and $\Gamma_*(IA_G)$ ? 
\end{reppb}

\begin{ex}
Consider the alternating group $A_n$. When $n \neq 2, 6$, $\Aut(A_n) =\Sigma_n$ (acting by conjugation), as is easily deduced from \cite[cor. 7.5]{Rotman}. But if $n \geq 5$, then $A_n$ is perfect: $[A_n,A_n] = A_n$. On the one hand, the Andreadakis filtration is thus constant equal to $IA(A_n) = \Aut(A_n) = \Sigma_n$. On the other hand, the lower central series of $\Sigma_n$ is  $\Sigma_n, A_n, A_n,...$, so the two filtrations differ in this case.
\end{ex}
Recall from the Introduction that we are interested in a stable form of the problem for $G=F_n$:
\begin{reppb}{pb_Andreadakis_stable}[Andreadakis - stable version]
What is the difference between $\mathcal A_k(F_n)$ and $\Gamma_k(IA_n)$ for $n \gg k$ ?
\end{reppb}

The \emph{Johnson morphism} turns out to be a powerful tool in the study of the Andreadakis filtration (which is the analogous of \emph{Johnson's filtration} on the mapping class group, defined by Johnson in \cite{Johnson}).

Recall the $\mathcal A_*(G)$ acts on $\Gamma_*(G)$ (see Proposition \ref{Universal_action}). The associated Johnson morphism $\tau$ is described in Example \ref{déf_tau}.
The $\Gamma_i(G)$ are characteristic subgroups of $G$, so $\mathcal A_0(G) = \Aut(G)$. Lemmas \ref{Johnson_equiv} and \ref{Johnson_inj} then give:	        
\begin{lem}\label{tau_inj}
The morphism $\tau$ is injective and $\Aut(G)/IA_G$-equivariant.
\end{lem}
The filtration $\Gamma_*(IA_G)$ also acts on $\Gamma_*(G)$, by pulling back the universal action along the inclusion $i: \Gamma_*(IA_G) \subseteq \mathcal A_*(F_n)$. The associated Johnson morphism is denoted by: 
$\tau' =\tau \circ i_*\label{déf_tau'}$.
The morphism $\tau'$ is still equivariant, but is injective if and only if $G$ satisfies Andreadakis' conjecture (\emph{i.e.\ }if $i_* = \Id{}$).

\begin{ex}\label{Johnson_Fn}
If $G $ is a free group, then $\Lie(F_n)$ is the free Lie algebra $\mathfrak LV$ on $V = G^{ab}$. If $\mathfrak g \circlearrowright \mathfrak h$ is an action in $\mathcal Lie$, then $\mathfrak g$ also acts on $\mathfrak h^o$, which is obtained by considering the module $\mathfrak h$ as an abelian Lie algebra. Derivations from $\mathfrak g$ to $\mathfrak h$ then identify with sections of the projection $\mathfrak h^o \rtimes \mathfrak g \twoheadrightarrow \mathfrak g$. Hence, the free Lie algebra is also free with respect to derivations. In particular:
\[\Der_k(\mathfrak LV) \cong \Hom_{\k}(V, \mathfrak L_k V) \cong V^*\otimes \mathfrak L_k V.\]
\end{ex}

The following result is well-known \cite[th.\ 6.1]{Kawazumi1}:
\begin{prop}\label{IAn^ab} 
In degree one, the Johnson morphism $\tau'$ is a $\GL_n(\mathbb Z)$-equivariant isomorphism:
\[\tau_1': IA_n^{ab} \cong V^*\otimes \Lambda^2V.\]
\end{prop}

\begin{proof}
The group $IA_n$ is generated by the following elements \cite{Nielsen} -- see also \cite[5.6]{BBM}:
\begin{equation}\label{gen_of_IAn}
K_{ij}: x_t \longmapsto \begin{cases}
                                 x_j x_i x_j^{-1} &\text{if } t = i \\
						                     x_t              &\text{else }    
													 \end{cases}
\ \ \text{ and }\ \ 
K_{ijk }: x_t \longmapsto \begin{cases}
                                [x_j, x_k] x_i &\text{if } t = i \\
						                     x_t              &\text{else. }    
														\end{cases}
\end{equation}
One can check by a direct calculation that these generators are sent to a basis of the free abelian group $V^*\otimes \Lambda^2V$.
\end{proof}

\subsection{Lazard's theorem}

\begin{defi}\label{def_filtered_alg}
A \emph{filtered algebra} $A_*$ is an associative $\k$-algebra $A_0$ endowed with a filtration by ideals: $A = A_0 \supseteq A_1 \supseteq ...$ such that:
$\forall i, j,\ A_i A_j \subseteq A_{i+j}.$ We denote by $\kfAlg[]$ the category of filtered algebras (and filtration-preserving morphisms).
\end{defi}

\begin{ex}
Let $\k G$ be the \emph{group algebra} of $G$ with coefficients in the (commutative) ring $\k$. We denote by $\varepsilon: \k G \rightarrow \k$ its canonical \emph{augmentation} and by $I_\k G:= \ker(\varepsilon)$ its \emph{augmentation ideal} (we will sometimes write $IG$, or even $I$, for short). Then $\k G$ is filtered by the powers $I^*G$ of its augmentation ideal. If $G$ is a free group, then $\gr(\k G)$ is the tensor algebra over $G^{ab} \otimes \k$ \cite[th.\ 6.2]{Passi}.
\end{ex}

From Theorem \ref{Kaloujnine}, we can deduce the useful corollary \cite[th.\ 3.1]{Lazard}:
\begin{theo}[Lazard]\label{Lazard}
Let $A = A_0 \supset A_1 \supseteq \cdots$ be a filtered algebra. 
Then $A^\times_*:= A^\times \cap (1 + A_*)$ is a strongly central filtration on $A^\times_1 \subseteq A^\times$, and $(-) -1$ induces an embedding of graded Lie algebras:
\[\Lie(A^\times_*) \hookrightarrow \gr_*(A_1).\]
\end{theo}

\begin{rmq}\label{Lazard_on_G}
Fix a morphism $G \overset{\alpha}{\longrightarrow} A^\times$. We can pull back the filtration given by the theorem to get a strongly central filtration $\alpha^{-1}(1+A_*)$ on $G_1 = \alpha^{-1}(1+A_1)$.
\end{rmq}

\begin{proof}[Proof of Theorem \ref{Lazard}]

The filtration $A = A_0 \supset A_1 \supseteq \cdots$ can be seen as a filtration of the abelian group $A$. As $A$ is abelian, it has to be strongly central. 
Consider the action by left multiplication:
\[\rho: A^\times \longrightarrow \Aut(A,+).\]
Let $a \in A^\times$. One can easily check that $\rho(a) \in \mathcal A_j(\cpx A)$ (where $\mathcal A_j(\cpx A)$ is the filtration defined in Kaloujnine's theorem \ref{Kaloujnine})  if and only if $a \in 1+A_j$. Thus, $A^\times \cap (1+A_i) = \rho^{-1}(\mathcal A_j(\cpx A))$ is a strongly central filtration, as announced. It remains to show that $\partial = \alpha - 1$ induces a morphism of Lie ring (necessarily injective). This can be checked directly from the formula:
\[[g,h] - 1 = [g-1, h-1]g^{-1}h^{-1}.\]
We will give a slightly different proof later on, using the concept of derivations (see Paragraph \ref{derivations}).
\end{proof}

\begin{ex}\label{dimension}
Applying Remark \ref{Lazard_on_G} to the inclusion of $G$ in $\k G$ filtered by the powers of the augmentation algebra, we get the \emph{dimension series} of $G$: 
\[D_*^{\k}G = G \cap (1+ I^*_{\k}G).\] 
It is a strongly central series on $G$, so it contains $\Gamma_*G$. The question of the equality of $D_*^{\Z}G$ and $\Gamma_*G$ was known as the \emph{dimension subgroup problem} during a long time, until an example of a group for which the two filtration differ was given in \cite{Rips}. See \cite[chap. 2]{Mikhailov} for more on this subject.

If $G$ is a free group, then $\Lie(D_*^{\Z}G)$ is the sub-Lie ring generated in degree one in the tensor algebra $\gr(\Z G) \cong TV$. Hence (by the PBW theorem), it is the free Lie ring. It then has to coincide with $\Lie(G)$, so $D_*^{\Z}G = \Gamma_*G$. Thus, free groups have the \emph{dimension property}.
\end{ex}

Lazard's theorem gives a construction of a strongly central filtration from a filtered algebra. Conversely, we can define a filtered algebra from a strongly central filtration $G_*$ on $G = G_1$. Indeed, let $\k G$ be filtered by:
\[F_i:= \k G \cdot (N_i-1) = \ker(\k G \longrightarrow \k (G/H_i)).\]
This filtration does not make $\k G$ into a filtered algebra, but it generates a filtration which does: 
\[\mathfrak a_j^\k(N_*):= \sum\limits_{i_1 + \cdots + i_n \geq j}{\k G \cdot (N_{i_1}-1) \cdots (N_{i_n}-1)}.\]
One can easily check that these constructions are universal:
\begin{prop}\label{adjonction_a}
The above constructions define an adjunction:
\[\begin{tikzcd}
  \mathcal{SCF} \ar[r, shift left, "\mathfrak a_*^{\k}"] & \kfAlg.  \ar[l, shift left, "(-)^\times_*"]
\end{tikzcd}\]
\end{prop}

\subsection{Congruence groups}\label{congruence_filtration}

If $I$ is an (associative) ring without unit, recall that its \emph{congruence group} $GL_n(I)$ is defined as:
\[GL_n(I):= \ker(GL_n(A) \longrightarrow GL_n(A/I)),\]
where $A$ is any unitary (associative) ring containing $I$ as a (two-sided) ideal, (\emph{e.g.\ } $A = I \rtimes \Z$). This group depends only on $I$, as it is exactly $(1 + M_n(I))^\times$.

\medskip

If $A = A_0 \supseteq A_1 \supseteq \cdots$ is a filtered algebra (see Definition \ref{def_filtered_alg}), then so is the matrix algebra $M_n(A)$, endowed with the filtration $M_n(A_*)$. Theorem \ref{Lazard} gives us a strongly central filtration of the congruence group:
\[GL_n(A_1) = GL_n(A) \cap (1 + M_n(A_1)) 
            = \ker(GL_n(A) \longrightarrow GL_n(A/A_1)).\]
by congruence subgroups:
\[GL_n(A_j) =  GL_n(A) \cap (1 + M_n(A_j)) 
            = \ker(GL_n(A) \longrightarrow GL_n(A/A_j)),\]
and an embedding of the associated Lie ring into a matrix algebra:
\[\Lie(GL_n(A_*)) \hookrightarrow \gr_*(M_n(A_*)) \cong M_n(\gr_*(A_*)).\]
As in the proof of Theorem \ref{Lazard}, this filtration can be interpreted as:
\[GL_n(A_*) = \mathcal A_*(GL_n(A), A^n_*).\]
We also can recover Lazard's theorem as the case $n = 1$ of this construction.

\medskip

Suppose that $A_*$ is commutative. Then the usual determinant defines a filtration-preserving morphism:
\[\det: GL_n(A_*) \longrightarrow GL_1(A_*) = A^\times_*.\]
Indeed, if $M \in M_n(A_j)$, then $\det(\Id{} + M) \in 1 + A_j$. The following proposition determines the associated graded morphism:

\begin{prop}\label{det_and_tr}
The following square commutes:
\[\begin{tikzcd} 
\Lie(GL_n(A_*)) \ar[r, hook, "(-) - \Id{}"] \ar[d, "\det"] & M_n(\gr(A_*)) \ar[d, "\tr"]\\
\Lie(A^\times_*) \ar[r, hook, "(-) - 1"]                & \gr(A_*).
\end{tikzcd}\]
Moreover, it is Cartesian, that is:
\[\Lie(GL_n(A_*)) \cong  \tr^{-1}(\Lie(A^\times_*) - 1).\]
\end{prop}

The kernels of $\det$ and $\tr$ then coincide. We thus recover (an generalize slightly) a result of \cite{Lopez}:
\begin{cor}\label{gr_of_sl_n}
Let $SL_n(A_*)$ be the kernel of the determinant. Then:
\[\Lie(SL_n(A_*)) \cong  \ker(\tr) = \mathfrak{sl}_n(\gr(A_*)).\]
\end{cor}

\begin{rmq}\label{SL_n_et_GL_n}
If $\Lie(A_*^\times) = 0$, then $\Lie(GL_n(A_*)) = \Lie(SL_n(A_*))$, and this Lie ring identifies to $\mathfrak{sl}_n(\gr(A_*))$. This happens for example when $GL_1(A_1) = \{1\}$ (implying $SL_n(A_*) = GL_n(A_*)$), which is verified for $A_* = q^*\Z$ (if $q > 2$), or $A_* = t^*\k[t]$. 
\end{rmq}

\begin{proof}[Proof of Proposition \ref{det_and_tr}]
Let $M \in M_n(A_j)$. Then:
\[\det(\Id{} + M) \equiv 1 + \tr(M) \pmod{A_j^2}.\]
When $j \geq 1$, then $A_j^2 \subseteq A_{j+1}$, so this formula gives the commutativity of the above square.
The module $\tr^{-1}(\Lie_j(A^\times_*) - 1)$ is additively generated by the matrices: 
\[\bar{a} \textbf{e}_{\alpha \beta},\ 
\bar{a} (\textbf{e}_{11} - \textbf{e}_{\alpha\alpha}) \text{ and } 
\bar{b}\textbf{e}_{11},\ \text{ for }
\alpha \neq \beta, a \in A_j \text{ and }1 + b \in A^\times_j.\] 
Replacing $\bar{a} (\textbf{e}_{11} - \textbf{e}_{\alpha\alpha})$ by $\bar{a} (\textbf{e}_{11} + \textbf{e}_{1\alpha} - \textbf{e}_{\alpha 1} - \textbf{e}_{\alpha\alpha})$, we can lift these to $GL_n(A_j)$ as follows:
\[\left\{
\begin{array}{lll}
\Id{} + a \textbf{e}_{\alpha \beta} &\text{ lifts }&\bar{a} \textbf{e}_{\alpha \beta}, \\
\Id{} + a (\textbf{e}_{11} + \textbf{e}_{1\alpha} - \textbf{e}_{\alpha 1} - \textbf{e}_{\alpha\alpha})  &\text{ lifts }
&\bar{a} (\textbf{e}_{11} + \textbf{e}_{1\alpha} - \textbf{e}_{\alpha 1} - \textbf{e}_{\alpha\alpha}), \\
\Id{} +b\textbf{e}_{11} &\text{ lifts }&\bar{b}\textbf{e}_{11}.
\end{array}
\right.\]
This completes the proof.
\end{proof}

Let $n \geq 3$. If $A$ is a "\emph{non-totally-imaginary Dedekind ring of arithmetic type}" and $\mathfrak q$ is an ideal of $A$, then we have \cite[cor. 4.3 (b)]{BMS} that $SL_n(\mathfrak q)$ is normally generated in $SL_n(A)$ (in fact in $E_n(A)$) by the shear mappings $\Id{} + t \textbf{e}_{\alpha\beta}$ with $\alpha \neq \beta$ and $t \in \mathfrak q$. This applies for instance to $A = \Z$ and $\mathfrak q = (q)$. 
From this we deduce:
\begin{prop}\label{LCS_of_sl_n} 
If $n \geq 5$, under the above hypothesis:
\[\Gamma_*(SL_n(\mathfrak q)) = SL_n(\mathfrak q^*).\]
\end{prop}

\begin{proof}
The filtration $SL_n(\mathfrak q^*)$ is strongly central on $SL_n(\mathfrak q)$, so it contains its lower central series. 
Conversely, using that:
\[\left\{
\begin{array}{ll}
\Id{} + (a+b)\textbf{e}_{\alpha\beta} = (\Id{} + a\textbf{e}_{\alpha\beta})(\Id{} + b\textbf{e}_{\alpha\beta}) &\text{if } \alpha \neq \beta, \\
\Id{} + ab\textbf{e}_{\alpha\beta} = [\Id{} + a\textbf{e}_{\alpha\gamma}, \Id{} + b\textbf{e}_{\gamma\beta}] &\text{if }\alpha, \beta \text{ and } \gamma \text{ are pairwise distinct},\\
\end{array}
\right.\]
one can easily check that for any $t$ in $\mathfrak q^k$ and any $\alpha \neq \beta$, if $n \geq 5$: 
\[\Id{} + t \textbf{e}_{\alpha\beta} \in \Gamma_k(SL_n(\mathfrak q)).\]
Using the result from \cite{BMS}, we see that these generate $SL_n(\mathfrak q^k)$ as a normal subgroup of $SL_n(A)$. Hence $SL_n(\mathfrak q^k) \subseteq \Gamma_k(SL_n(\mathfrak q))$, as required.
\end{proof}

\begin{rmq}[On the $q$-torsion Andreadakis problem for $\Z^n$]
Fix an integer $q$. A \emph{$q$-torsion} strongly central filtration is a strongly central filtration $G_*$ such that $G_i^q \subseteq G_{i+1}$ for every $i \geq 1$ (this means exactly that $q\Lie(G_*) = 0$). On any group $G$, there is a minimal $q$-torsion strongly central filtration $\Gamma^{(q)}_*(G)$. Moreover, $\mathcal A_*(G_*)$ is $q$-torsion whenever $G_*$ is, because $\Lie(\mathcal A_*(G_*))$ embeds into a $q$-torsion Lie algebra \emph{via} the Johnson morphism. Let us denote $\mathcal A_*(\Gamma^{(q)}_*(G))$ by $\mathcal A^{(q)}_*(G)$ and $\mathcal A^{(q)}_1(G)$ by $IA^{(q)}(G)$ (it is the group of automorphisms acting trivially on $G^{ab} \otimes (\Z/q)$). Thus we get an inclusion of $\Gamma^{(q)}_*(IA^{(q)})$ into $\mathcal A_*^{(q)}$ and a corresponding \emph{$q$-torsion Andreadakis problem}.

Apply this with $G =\Z^n$ and $q \geq 3$. Then $\Gamma^{(q)}_*(G) = q^*\Z$, $\Aut(G) = GL_n(\Z)$, and $\mathcal A^{(q)}_*(G) = GL_n(q^*\Z) = SL_n(q^*\Z)$
Proposition \ref{LCS_of_sl_n} then gives an answer to the $q$-torsion Andreadakis problem for $\Z^n$ in the stable range ($n \geq 5$) : $\mathcal A^{(q)}_*(\Z^n)$ is the lower central series of $IA^{(q)}(\Z^n) = SL_n(q\Z)$.
\end{rmq}

When $A = \Z$ and $\mathfrak q = (q)$, the graded ring $\gr(q^*\Z)$ is $(\Z/q)[t]$. Moreover, $GL_n(q\Z) = SL_n(q\Z)$ (see Remark \ref{SL_n_et_GL_n}). Proposition \ref{LCS_of_sl_n} and Corollary \ref{gr_of_sl_n} thus give:
\begin{cor}\label{Lie(GL_n(qZ))}
For all $n \geq 5$ and all $q \geq 3$, there is a canonical isomorphism of graded Lie rings (in degrees at least one):
\[\Lie(GL_n(q\Z)) \cong  \mathfrak{sl}_n(\Z/q)[t],\]
where the degree of $t$ is $1$, and the Lie bracket of $Mt^i$ and $Nt^j$ is $[M,N]t^{i+j}$.
\end{cor}

\begin{rmq}
This generalizes \cite[Th.\ 1.1]{Lee-Szczarba}, which is the degree-one part.
\end{rmq}

\begin{rmq}
For $n =2$ and $q \geq 5$ a prime number, the group $SL_2(q\Z)$ is free on $1+ q(q^2-1)/12$ generators \cite{Grosswald, Frasch}. Its Lie ring is then a free Lie ring on the same number of generators. The author does not know a complete calculation for $n = 3$ or $4$ : the above calculus does give the abelianization, but it fails to determine the whole lower central series.
\end{rmq}

\subsection{Comparison between filtrations obtained from an action}

Let $G$ be a group. Suppose that $G$ acts on two strongly central series $H_*$ et $K_*$ (by automorphisms in $\mathcal{SCF}$). We then can ask what link exists between $\mathcal A_*(G,H_*)$ and  $\mathcal A_*(G,K_*)$ (as defined in Remark \ref{A_on_K}), depending on the links between $H_*$ an $K_*$. 

The next proposition describes the behaviour of the construction $\mathcal A_*(G,-)$ with respect to injections, surjections (see Definition \ref{def_inj-surj}) and semi-direct products in the category $G-\mathcal{SCF}$ of strongly central series endowed with a $G$-action (where morphisms respect this action):

\begin{prop}\label{inj-surj} Let $G$ be a group acting on strongly central series $N_*$, $H_*$ and $K_*$. Let $u: N_* \longrightarrow H_*$ and $v: H_* \longrightarrow K_*$ be $G$-equivariant morphisms.

\smallskip

If $u: N_* \longrightarrow H_*$ is an injection, then $\mathcal A_*(G,H_*) \subseteq \mathcal A_*(G,N_*)$.

If $v: H_* \longrightarrow K_*$ is a surjection, then $\mathcal A_*(G,H_*) \subseteq \mathcal A_*(G,K_*)$.

If  $\begin{tikzcd}[column sep=small]
                   N_* \ar[r, hook] & H_* \ar[r, two heads] & K_* \ar[l, bend right]
                  \end{tikzcd}$
is a split exact sequence in $G-\mathcal{SCF}$, then:
\[\mathcal A_*(G,H_*) = \mathcal A_*(G,K_*) \cap \mathcal A_*(G,N_*).\]
\end{prop}

\begin{proof}
In the first case, we identify $N = N_1$ to a subgroup of $H = H_1$. Let $g \in \mathcal A_j(G,H_*)$. We write:
\[[g,N_i] = [g, N \cap H_i] \subseteq [g,N] \cap [g, H_i] \subseteq N \cap H_{i+j} = N_{i+j}.\]
Similarly, to show the second assertion, let us take $g \in A_j(G,H_*)$. We write:
\[[g,K_i] = [g, v(H_i)] = \varphi([g, H_i]) \subseteq v(H_{i+j}) = K_{i+j}.\]
The third assertion's hypothesis comes down to require that $H_*$ decompose as a semi-direct product $K_* \rtimes N_*$, the action of $G$ on $H_*$ being factor-wise. We then get $G$-equivariant isomorphisms:
\[H_i/H_{i+j} \cong K_i/K_{i+j} \rtimes N_i/N_{i+j}.\]
An element $g$ of $G$ acts trivially on the left hand side if and only if it does on the right hand side. Whence the result.
\end{proof}

\begin{rmq}
Let $\begin{tikzcd}[column sep=small]
                  1 \ar[r] & N_* \ar[r, "u"] & H_* \ar[r, "v"] & K_* \ar[r] & 1
                  \end{tikzcd}$ be a non-split short exact sequence in $G-\mathcal{SCF}$, the first part of Proposition \ref{inj-surj} gives:
\[\mathcal A_*(G,H_*) \subseteq \mathcal A_*(G,K_*) \cap \mathcal A_*(G,N_*).\]
Nevertheless, equality is not true in general. Indeed, the sequences:
\[\begin{tikzcd}[column sep=small]
                  1 \ar[r] & N_i/N_{i+j} \ar[r, "u"] & H_i/H_{i+j} \ar[r, "v"] & K_i/K_{i+j} \ar[r] & 1
                  \end{tikzcd}\]
are exact, but $g \in G$ can act trivially on the kernel and quotient without acting trivially on the middle term. For instance, $H^1(G) = \Ext^1_G(\Z^{triv}, \Z^{triv})$ is non-trivial in general.
\end{rmq}

%[QUESTION: ceci donne un exemple pour les groupes abéliens, mais quid du cas où $G =\mathcal A_1$ ? Peut-on montrer un résultat plus faible, mais meilleur que les produits semi-directs ?]

We can get a little more about semi-direct products:

\begin{prop}\label{transitivity_of_A}
Let  $\begin{tikzcd}[column sep=small]
                   N \ar[r, hook] & H \ar[r, two heads] & K \ar[l, bend right]
                  \end{tikzcd}$
be a split exact sequence in $G - \mathcal Grps$. Suppose that $N$ is filtered by a strongly central series $N_*$. Then:
\[\mathcal A_*(G,N_*) \subseteq \mathcal A_*\left(G,\mathcal A_*(K, N_*)\right).\]
\end{prop}

\begin{proof}
Let us denote by $K_*$ the filtration $\mathcal A_*(K, N_*)$ and by $G_*$ the filtration $\mathcal A_*(G, N_*)$. A straightforward application of the 3-subgroup lemma (Lemma \ref{lem3sgrp}) in $(N \rtimes K) \rtimes G$ provides the inclusion:
\[[[G_\alpha, K_\beta], N_\gamma] \subseteq N_{\alpha + \beta +\gamma}.\]
Thus $[G_\alpha, K_\beta] \subseteq \mathcal A_{\alpha + \beta}(K, N_*) = K_{\alpha + \beta}$, 
which means that $G_\alpha \subseteq \mathcal A_\alpha (G,K_*)$.
\end{proof}

\begin{cor}\label{incl_de_A}
Let $G$ be a group acting on a filtered algebra $A_*$ by automorphisms of filtered algebras. Then:
\[\mathcal A_*(G,A_*) \subseteq \mathcal A_*(G,A^\times_*).\]
\end{cor}

\begin{proof}
Apply Proposition \ref{transitivity_of_A} to the $G$-equivariant split exact sequence:
\[\begin{tikzcd}[column sep=small]
                   A \ar[r, hook] & A \rtimes A^\times \ar[r, two heads] & A^\times \ar[l, bend right]
                  \end{tikzcd}\]
and the given filtration on $A$.
\end{proof}

Here is an interesting case when the filtrations of Corollary \ref{incl_de_A} are equal:
\begin{prop}\label{égalité_de_A}
Let $G$ be a group, and $\k$ a commutative ring. Then:
\[\mathcal A_* \left(\Aut(G), D_*^\k G\right) = \mathcal A_* \left(\Aut(G), I_\k^*G\right).\]
\end{prop}

\begin{proof}
The algebra $\Z G$ is filtered by $I_\k^*G$, the powers of its augmentation ideal. The group $\Aut(G)$ acts on $\Z G$ by automorphisms of filtered algebras.
As $\mathcal A_*(G, I_\k^*G) = D_*^\k G$, Proposition \ref{transitivity_of_A}, applied to the action of $\Aut(G)$ on $\k G \rtimes G$, gives an inclusion:
\[\mathcal A_*(\Aut(G), I_\k^*G) \subseteq \mathcal A_*(\Aut(G), D_*^\k G).\]
To show that it is in fact an equality, take $\varphi \in \mathcal A_*(\Aut(G), D_*^\k G)$. Then:
\[\forall g \in G = D_1G,\ \ [\varphi, g] = (\varphi(g)g^{-1} - 1)g \in (D_{j+1} - 1)G \subseteq I^{j+1}.\]
We then show that $[\varphi, I^i] \subseteq I^{i+j}$ by induction on $i \geq 1$,
using the formula:
\[[\varphi,uv] = [\varphi,u]\varphi(v) + u[\varphi,v]\ \ \left(=(\varphi(u)-u) \cdot \varphi(v) + u \cdot (\varphi(v)-v)\right).\]
Let us remark that, in the language of paragraph \ref{alg_actions_and_der}, the last formula states that $[\varphi,-]$ is a $(\Id{},\varphi)$-derivation, so we have in fact used lemma \ref{degree_of_derivations}.
\end{proof}

\section{Traces and stable surjectivity}\label{Section_stable_surj}

\subsection{Free differential calculus}

We recall some basic concepts of free differential calculus. A detailed account can be found in \cite{Fox}.

\begin{defi}\label{def_derivations_ab}
Let $G$ be a group, and $M$ a $\k G$-module. A \emph{derivation} from $G$ to $M$ is a map $\partial: G \longrightarrow M$ such that:
\[\forall g, h \in G,\ \partial(gh) = \partial g + g \cdot \partial h.\]
It can be extended to a linear map $\partial: \k G \longrightarrow M$, which verifies:
\[\forall u, v \in \k G,\ \partial(uv) = \partial (u) \varepsilon (v) + u \cdot \partial (v).\]
We denote by $\Der(G,M)$ or $\Der(\k G, M)$ the space of derivations from $G$ to $M$. We will often write $\Der(\k G)$ for $\Der(\k G, \k G)$. 
\end{defi}

\begin{rmq}
Let $G = F_S$ be the free group over a set $S$. As derivations identify with sections of $M \rtimes G\twoheadrightarrow G$, we get:
$\Der(\k F_S, M) \cong M^S$ for any $\k G$-module $M$.
\end{rmq}

\begin{defi}
Let $S = (x_i)_i$ be a chosen basis of a free group $F$. The following requirement defines a derivation of $\k F$:
\[\frac{\partial}{\partial x_i}: x_t \longmapsto \begin{cases}
                                               1 &\text{ if } t = i, \\
						                       0 &\text{ else.}    
													 \end{cases}\]
\end{defi}

Let us give a first version of the chainrule:

\begin{prop}\label{chainrule}
Let
$\lambda: F_Y \longrightarrow G $ a group morphism, where $F_Y$ is the free group on a set $Y = \{y_j\}$.
Then, for $u$ in $\k Y$ and $\partial$ in $\Der(\k G)$: 
\[\partial (\lambda u) = \sum\limits_j \lambda \left( \frac{\partial u}{\partial y_j} \right) \partial (\lambda y_j).\]
\end{prop}

\begin{rmq}
The sums involved here are finite, because only a finite number of letters appear in a given element $u$. 
\end{rmq}

\begin{proof}[Proof of proposition \ref{chainrule}]
One can check that each member of this equality defines a derivation from $\k F_Y$ to $\k G$, where $\k F_Y$ acts on $\k G$ by $y \cdot g = \lambda(y)g$. As these formulas give the same result when evaluated at elements of the basis, the corresponding derivations are equal.
\end{proof}

If $G = F_X = \left\langle x_i \right\rangle$ is free too, we can apply proposition \ref{chainrule} with $\partial = \frac{\partial}{\partial x_i}$ to get:
\[\frac{\partial (\lambda u)}{\partial x_i}  = \sum\limits_j \lambda \left( \frac{\partial u}{\partial y_j} \right) \frac{\partial (\lambda y_j)}{\partial x_i}.\]
For $X = Y$ and $\lambda = \Id{F_X}$, this gives a change-of-base formula similar to the usual one.

\medskip

The following definition keeps on with our analogy with classical differential calculus:

\begin{defi}
Let $F_Y = \left\langle y_j \right\rangle$ and $F_X = \left\langle x_i \right\rangle$ be free groups as above. Let $f$ be a morphism from $F_Y$ to $F_X$. We define its \emph{Jacobian matrix} (with respect to the chosen basis) by:
\[D(f):= \left( \frac{\partial f(y_j)}{\partial x_i} \right)_{ji} \in M_{YX}(\k X).\]
\end{defi}

\begin{rmq}
The morphism $f$ is determined by $Df$. Indeed, proposition \ref{chainrule}, applied with $\lambda = 1$, $\partial: v \longmapsto v - \varepsilon(v)$ and $u = f(x_i)$, gives:
\[f(x_i) - 1 = \sum\limits_j \frac{\partial f(x_i)}{\partial x_j} (x_j - 1).\]
\end{rmq}

Let
$F_Z = \left\langle z_k \right\rangle$, 
$F_Y = \left\langle y_j \right\rangle$ and 
$F_X = \left\langle x_i \right\rangle$ be free groups, and let
$F_Z \overset{g}{\longrightarrow} F_Y \overset{f}{\longrightarrow} F_X$ be morphisms between them. We can use proposition \ref{chainrule} to get a chainrule for Jacobian matrices:

\[\frac{\partial f(g(z_k))}{\partial x_i}  = \sum\limits_j f \left( \frac{\partial g(z_k)}{\partial y_j} \right) \frac{\partial f(y_j)}{\partial x_i}.\]
This can be restated as:
\begin{cor}\label{Dérivation_composée}
Let $F_Z \overset{g}{\longrightarrow} F_Y \overset{f}{\longrightarrow} F_X$ be morphisms between free groups with fixed basis, as above. Then:
\[D(fg) = f(Dg)D(f).\]
\end{cor}

\medskip

\begin{rmq}\label{sens_D(fg)}
The reader may have noticed that this formula seems to come "in the wrong way". This can be explained as follows: a morphism $f: F_Y \longrightarrow F_X$ is in fact a $r$-tuple of monomials $f(y_j) = f_j(x_i) \in \k F_X$, and should as such be considered as a "polynomial function from $F_X$ to $F_Y$", whose coordinates would be the $f_j$s. From this point of view, $fg$ would be a "polynomial function from $F_X$ to $F_Z$", whose coordinates would be given by $fg(z_k) = g_k(f_j(x_i))$: it looks more like "$g \circ f$" !

This seems to be completely analogous to the classical setting of algebraic geometry. To get more accurate statements, one would have to interpret $\k F_X$ as an algebra of functions over a geometric object associated to $F_X$. Which should look like a one-point object with some local structure (as $Df$ determines $f$).
\end{rmq}

\subsection{Derivations and strongly central filtrations}\label{derivations}

We introduce the notion of a derivation from a group $G$ to a group $H$ on which it acts. Our aim here is to describe a general framework which will be useful to study Jacobian matrices and their interactions with Lie brackets: the chainrule formula \ref{Dérivation_composée} tells us that $D$ is a derivation. We will also get back to Lazard's theorem \ref{Lazard} in this framework.

\begin{defi}\label{déf_dérivations}
Let $H$ be a group, on which another group $G$ acts. A map $\partial: G \longrightarrow H$ is a \emph{derivation} if:
\[\forall\ x, y \in G,\ \partial (xy) = \partial x \cdot {}^x\!\partial y.\]
\end{defi}

\begin{rmq}
If $H = M$ is an abelian group, \emph{i.e.\ }a representation of $G$, then we recover the usual definition of a derivation from $G$ to $M$ (see Definition \ref{def_derivations_ab}).
\end{rmq}

To give a derivation $\partial: G \longrightarrow H$ is exactly the same as giving a section $\sigma = (\partial,\Id{G})$ of the canonical projection:
\[\begin{tikzcd}
  H \rtimes G \ar[r, swap, pos=0.4, "p"] &
  G.           \ar[l, bend right, dashed, swap, pos=0.6, "\sigma"]
  \end{tikzcd}\]
Keeping this in mind, the following lemma follows immediately:

\begin{lem} \label{sous-rep}
Let $\partial$ be a derivation from $G$ to $H$. Then $\partial^{-1}$ sends $G$-stable subgroups of $H$ on subgroups of $G$. 
\end{lem}

\medskip

Let $H_*$ be strongly central filtration on a subgroup $H_1$ of $H$. Let $G_*$ be a strongly central filtration on a subgroup $G_1$ of $G$, which acts on $H_*$ through the given action of $G$ on $H$ (see Proposition \ref{actions_dans_SFC}).
A derivation $\partial$ from $G$ to $H$ being given, we can use the morphism $\sigma = (\partial,\Id{G})$ to pull back the filtration $H_* \rtimes G_*$. We thus get a strongly central filtration on $G_1 \cap \partial^{-1}(H_1)$:
\[\sigma^{-1}(H_i \rtimes G_i) = G_i \cap \partial^{-1}(H_i).\]

\begin{rmq}\label{SCF_sur_G}
For instance, if $H_*$ is given, we can let $G_*$ be $\mathcal A_*(G, H_*)$, the maximal filtration acting on $H_*$, as described in Remark \ref{A_on_K}. The above construction then gives a strongly central filtration on $\mathcal A_1 \cap \partial^{-1}(H_1)$. This subgroup is all of $G$ if and only if:
\begin{equation}
\left\{
\begin{array}{l}
G \text{ stabilises } \cpx H,\\
G \text{ acts trivially on } \Lie(\cpx H),\\
\partial(G) \subseteq H_1.
\end{array}
\right.
\end{equation}
Under these conditions, $\mathcal A_* \cap \partial^{-1}(\cpx H)$ is a strongly central series on $G$. In particular, it then contains $\cpx\Gamma(G)$.
\end{rmq}

\medskip

Keeping the above notations, the morphism $\sigma = (\partial,\Id{G})$ induces a Lie ring morphism (which is injective by definition of the filtration on the domain):
\[\bar\sigma: \Lie(G_* \cap \partial^{-1}(\cpx H)) \longhookrightarrow \Lie(\cpx H) \rtimes \Lie(G_* ).\]
This ensures that $\partial$ induces a well-defined linear map $\bar\partial$ between the Lie algebras. Moreover, the map $\bar\sigma = (\bar\partial, \bar{\Id{}}): x \longmapsto \bar\partial x + x$ preserves Lie brackets, hence:
\[\bar\partial([x,y]) = [\bar\partial x,y] + [x,\bar\partial y] + [\bar\partial x,\bar\partial y].\]
If $\Lie(H_*)$ is an abelian Lie algebra, then the last term is zero, and $\bar\partial$ is a Lie derivation. This happens in particular when $H$ is an abelian group.

\medskip

\begin{proof}[Back to the proof of Lazard's theorem \ref{Lazard}]
Take the filtered (abelian) group $(A,+)$ as $H$, and the group $A^\times$ as $G$ acting by left multiplication $\rho$. We already know that:
\[\mathcal A_j(A^\times, \cpx A) = A^\times_* = A^\times \cap (1 + A_j).\] 
Let $\partial$ be the derivation from $A^\times$ to $A$ defined by $g \longmapsto g-1$. 
Obviously, $\partial^{-1}(A_j) = A^\times_j$, which is exactly $\mathcal A_j(A^\times, \cpx A)$. It is then equal to $\mathcal A_j(A^\times, \cpx A)\cap \partial^{-1}(A_j)$. 

The Lie ring $\Lie(\cpx A)$ is abelian (because $A$ is an abelian group). Hence, the induced map $\bar\partial$ is a derivation (with respect to the canonical action of $\Lie(A^\times_*)$ on $\Lie(A_*)$):
\[\bar\partial: \Lie(A^\times_*) \longrightarrow \Lie(\cpx A).\]
Let us remark that the Lie algebra $\Lie(A_*)$ is quite different from $\gr(A_*)$: the associative structure of $A$ has been completely forgotten. Nevertheless, some part of this structures is encoded by the action of $\Lie(A^\times_*)$, which is inherited from left multiplication. The map $\bar\partial$ is a derivation with respect to this action, that is:
\[\bar\partial([x,y]) = [\bar\partial x,y] + [x,\bar\partial y].\]
These brackets are described through the action of $\Lie(A^\times_*)$ on $\Lie(\cpx A)$, as induced by commutators in $A \rtimes A^\times$:
\[\forall g \in A^\times_1,\ \forall x \in A,\ [g,x] = gx-x.\]
As a consequence:
\[\bar\partial([x,y]) = -(y(x-1) - (x-1)) + (x(y-1) - (y-1)) = xy-yx,\]
so $\bar\partial$ is in fact a Lie morphism to $\gr(A_*)$.
\end{proof}

\subsection{Algebras, actions and derivations}\label{alg_actions_and_der}

We now turn to studying derivations of algebras. In particular, we get a precise link between free differential calculus and differential calculus in the tensor algebra (see Proposition \ref{graded_of_derivations}). We will use this in Paragraph \ref{contraction_map} to get an explicit description of the trace map.

Let $\mathcal Alg_{-}$ be the category of associative non-unitary algebras over a fixed commutative ring $\k$. This category is pointed (by $0$) and protomodular. We can define actions there, as in paragraph \ref{Actions}. Actions in $\mathcal Alg_{-}$ turn out to be representable. Precisely, for any algebra $I$, let $\End_r(I)$ (resp.\ $\End_l(I)$) be the algebra of right (resp.\ left) $I$-linear endomorphisms of $I$, \emph{i.e.\ }$\k$-linear maps $u$ from $I$ to $I$ satisfying:
\[\forall x, y \in I,\ u(xy) = u(x)y \ \ (\text{resp.\ } u(xy) = x u(y)).\]
Define $\End_{r,l}(I)$ as the kernel of :
\[\alpha : \left \{\begin{array}{clc}
                   \End_r(I) \times \End_l(I) &\longrightarrow &\End_{\k}(I) \\
                   (u,v) &\longmapsto &u \circ v - v \circ u.
                   \end{array}
           \right.\]

\begin{prop}
An action $A \circlearrowright I$ in $\mathcal Alg_{-}$ can be represented by a (unique) morphism:
\[A \overset{(\lambda,\rho)} {\longrightarrow} \End_{r,l}(I).\]
\end{prop}

\begin{proof} 
If an action
\[\begin{tikzcd} I \ar[r, hook] & B \ar[r, two heads] & A \ar[l, bend right]
\end{tikzcd}\]
is given, $\lambda(a)$ and $\rho(a)$ are obtained from left and right multiplications by $a$ in $B$. Conversely, a morphism $(\lambda,\rho)$ as above can be used to define an (associative) algebra structure on $I \times A$ defining an action of $A$ on $I$.
\end{proof}
Let us remark that if $I^2 = 0$ (that is, $I$ is endowed with a trivial algebra structure), then an action of $A$ on $I$ is just a $A$-bimodule structure.

\begin{rmq}
The same construction works in the category of (non-unitary) filtered algebras $f\mathcal Alg_-$, where we also get a representation of actions. The algebras $\End_r(I)$ and $\End_l(I)$ are then filtered by the usual requirement: a morphism $u$ is of degree at least $j$ if $u(I_i) \subseteq I_{i+j}$ for all $i$. The same requirement will be used to define a filtration on any module of morphisms between filtered modules or algebras.
\end{rmq}

\begin{defi}
Let $A$ act on $I$ as above. A \emph{derivation} from $A$ to $I$ is a $\k$-linear map $\partial: A \longrightarrow I$ satisfying:
\[\partial(ab) = \partial a \cdot b +  a\cdot \partial b.\] 
The $\k$-module of derivations from $A$ to $I$ is denoted by $\Der(A,I)$. 
\end{defi}

\begin{rmq}\label{der_of_alg_as_sections}
The relation defining derivations depends only on the $A$-bimodule structure on $I$. We are thus led to consider $I^o$, the algebra obtained by taking the same underlying $\k$-module as $I$, endowed with the trivial product. The action of $A$ on $I$ induces an action of $A$ on $I^o$ in the obvious manner, and a derivation from $A$ to $I$ is then the same as a section of the projection:
\[I^o \rtimes A \twoheadrightarrow A\]
in the category of algebras.
\end{rmq}

When we work in the category of filtered algebras, $\Der(A,I)$ is a filtered module, a derivation $\partial$ being of degree at least $j$ if $\partial(A_i) \subseteq I_{i+j}$ for all $i$. If $A$ is filtered by its powers $A^i$, we just have to check this in degree one:

\begin{lem}\label{degree_of_derivations}
Let $A$ be an algebra, filtered by its powers $A_i:= A^i$, acting on a filtered algebra $I_*$. Let $\partial \in \Der(A,I)$. Then $\partial$ is of degree at least $j$ if and only if:
\[\partial(A) \subseteq I_{j+1}.\]
\end{lem}

\begin{proof}
An action of $A_*$ on $I_*$ is given by a left and a right multiplication which are filtered, meaning that $A_i I_j \subseteq I_{i+j}$ and $I_j A_i \subseteq I_{i+j}$. Use the formula:
\[\partial(a_1 \cdots a_i) = \sum\limits_k{a_1 \cdots a_{k-1} \partial(a_k)a_{k+1} \cdots a_i}\]
to get the desired result.
\end{proof}

We can get examples of actions from algebras acting on themselves. Precisely, the adjoint action of $A$ on itself is just the obvious $A$-bimodule structure on $A$. Derivations are then the usual ones.

Given an action of $A$ on $I$ represented by $(\rho, \lambda)$, we can twist it by choosing endomorphisms $\varphi$ and $\psi$ of $A$ and letting $A$ act on $I$ through $(\rho \circ \varphi, \lambda \circ \psi)$. This means that we let $a \in A$ act on $I$ by $\varphi(a)\cdot -$ on the left, and by $- \cdot \psi(a)$ on the right. We give a name to derivations from $A$ to the the twisted $A$-bimodule $I$.
\begin{defi}
Let $\varphi$ and $\psi$ be endomorphisms of $A$. A \emph{$(\varphi, \psi)$-derivation} is a linear map $\partial: A \longrightarrow I$ satisfying:
\[\partial(ab) = \partial a \cdot \psi b +  \varphi a \cdot \partial b.\]  
We denote by $\Der_{(\varphi, \psi)}(A,I)$ the $\k$-module of such derivations.
\end{defi}

\begin{ex}
Let $A$ be a group algebra $\k G$ and $M$ a $\k G$-module. We can make $M$ into a bimodule by making $\k G$ act trivially on the right (that is, through $\varepsilon$). Then, $\Der(\k G, M)$ is exactly the usual module of derivations (see Definition \ref{def_derivations_ab}). If $M = \k G$, it is already a bimodule, but the above structure can be obtained through twisting the right action by $\eta \varepsilon: g \mapsto \varepsilon(g) \cdot 1$. Then
$\Der(\k G)$ (defined in Definition \ref{def_derivations_ab}) is exactly $\Der_{(id, \eta\varepsilon)}(\k G, \k G).$
\end{ex}

We can apply Lemma \ref{degree_of_derivations} to $A = I = IG$, to get: 
\begin{cor}\label{degree_of_derivations_of_kG}
Let $\partial$ be a derivation of $\Z G$ such that $\partial (\Z G) \subseteq (IG)^{l+1}$ (which is always true for $l = -1$). For all integer $k$, we have:
\[\partial ((IG)^k) \subset (IG)^{k+l}.\]
\end{cor}

\begin{rmq}
Let us stress that the proof given here is fairly direct. In fact, it gets even shorter in this case, the result following from: 
$\forall v \in I,\ \partial(uv) = u \cdot \partial v.$
\end{rmq}

\begin{rmq}\label{derivation_on_LCD}
The classical inclusion of $\Gamma_k - 1$ into $IG^k$ can be shown by a direct induction, or follows from Lazard's theorem (as $D_*(G)$ is a strongly central series, it contains  $\Gamma_* G$). Under the hypothesis of Proposition \ref{degree_of_derivations_of_kG}, it implies
$\partial (\Gamma_k) \subseteq (IG)^{k+l}$,
using that $\partial(1) =0$. 
\end{rmq}

\begin{rmq}\label{algebra_structure_on_Der}
Some sets of derivations obtained from actions of $A$ on itself can have more structure than just a module structure. Precisely, if we twist the adjoint action of $A$ by some $(\varphi, \psi)$ as above, and if we add the requirement that $\varphi$ and $\psi$ are idempotents, then the set of $(\varphi, \psi)$-derivations from $A$ to $A$ commuting to $\varphi$ and $\psi$ is a sub-Lie algebra of $\End_{\k}(A)$:
\begin{align*}
[\partial, \partial'](ab) =\  &\partial\partial'a \cdot \psi^2 b
                               + \varphi\partial'a \cdot \partial\psi b
                               + \partial\varphi a \cdot \psi\partial' b
                               + \varphi^2 a \cdot \partial\partial' b \\
                             &- \partial'\partial a \cdot \psi^2 b
                               - \varphi\partial a \cdot \partial'\psi b
                               - \partial'\varphi a \cdot \psi\partial b
                               - \varphi^2 a \cdot \partial'\partial b \\
                          =\  &[\partial, \partial'](a) \cdot \psi(b) + \varphi(a) \cdot [\partial, \partial'](b).
\end{align*}
In particular, if $(\varphi, \psi) = (\Id{}, \Id{})$, we get that $\Der(A)$ is a sub-Lie algebra of $\End_{\k}(A)$. Another example is given by $A = \k G$ and $(\varphi, \psi) = (\Id{}, \eta \varepsilon)$. But more is true is this last case. Let $\k G$ be filtered by the powers of the augmentation ideal. If $\partial, \partial' \in \Der(\k G)$ are such that $\partial'$ has degree at least $0$, then $\partial \circ \partial' \in \Der(\k G)$, because $\varepsilon(\partial' v) = 0$ for any $v$.
\end{rmq}

Let $A_* \circlearrowright I_*$ be an action of filtered algebras. Since the functor $\gr: f\mathcal Alg \longrightarrow gr\mathcal Alg$ from filtered algebras to graded ones is exact (the same proof as that of Proposition \ref{exactness_of_L} works), this action is sent to an action $\gr(A_*) \circlearrowright \gr(I_*)$ of graded algebras. Moreover, $\gr$ also commutes with $(-)^o$ (the definition of $(-)^o$ being extended to graded algebras in the obvious way), so we get a morphism:
\[\gr(\Der(A_*, I_*)) \hookrightarrow \Der_*(\gr(A_*), \gr(I_*)),\]
where the target is the graded module of graded derivations. This morphism is obviously injective. It is in fact a restriction of the natural injection:
\[\gr(\Hom_{\k}(M_*, N_*)) \hookrightarrow \Hom_*(\gr(M_*), \gr(N_*))\]
between bifunctors on graded modules. As such, it preserves all algebraic structure inherited from the additive bifunctor structure (see Remark \ref{algebra_structure_on_Der}).

When $A_*$ is $\k G$, filtered by the powers of $IG$, acting on itself, we thus get a morphism preserving the structure induced by composition:
\[\gr(\Der(\k G)) \hookrightarrow \Der_*(\gr(\k G)).\]

\begin{prop}\label{graded_of_derivations}
If $G$ is a free group, and $M_*$ is a filtered $\k G$-module (considered as a bimodule with trivial right action), the canonical map:
\[\gr(\Der(\k G, M_*)) \hookrightarrow \Der_*(\gr(\k G), \gr(M_*))\]
is an isomorphism. Here, by derivations, we mean $(id, \varepsilon)$-ones.
\end{prop}

\begin{proof}
Let $S$ be a free set of generators for $G$. Then $V = G^{ab}$ is free abelian on $S$, and $\gr(\k G) \cong TV$ is the tensor algebra. Identifying derivations with sections as above (see Remark \ref{der_of_alg_as_sections}), we see that a derivation is completely determined by the choice of its values on $S$:
\[\Der_*(TV, N_*) = \mathcal F_*(S,N_*),\]
for any graded $TV$-bimodule $N_*$, where $\mathcal F_*(S,N_*)$ is the set of graded maps from $S$ (concentrated in degree $0$) to $N_*$.  The same is true for the other side. Indeed, a derivation from $\k G$ to $M$ is a section of the projection $M \rtimes G \twoheadrightarrow G$, so is determined by a map $S \rightarrow M$:
\[\Der(\k G, M_*) =\Der(G, M_*) \cong M_*^S.\]
The second member is the set of maps from $S$ to $M$, with the filtration inherited from the one on $M$. The desired isomorphism is then exactly:
$\gr(M_*^S) \cong \mathcal F_*(S,\gr(M_*)).$ 
\end{proof}

\begin{rmq}
If $M$ is a $G$-module, we can endow it with the universal $\k G$-filtration $(IG)^*\cdot M$.
\end{rmq}

\begin{rmq}
The isomorphism $\gr(\Der(\k G)) \cong \Der_*(TV)$ thus obtained preserves the algebraic structure obtained from the composition of derivations.
\end{rmq}

\subsection{Traces}

In \cite{Bartholdi}, Bartholdi defines the trace of an automorphism $\varphi$ of $F_n$ by:
\[\tr(\varphi):= \tr(D \varphi - \Id{}) \in \Z F_n,\]
where $D \varphi$ denotes $\varphi$'s Jacobian matrix.
We will show that $\tr$ induces a well-defined map between the graded Lie algebras, which we still call $\tr$:
\[\tr:\Lie(\mathcal A_*(F_n)) \longrightarrow \gr(\Z F_n) \cong TV.\]
The aim of this paragraph is to show that this map is indeed well-defined, to investigate its behaviour with respect to Lie structures, and to get Morita's algebraic description \cite{Morita}, used by Satoh in \cite{Satoh2}.

\subsubsection{The induced map between Lie algebras}\label{gradué}

Let $\varphi \in \mathcal A_k(F_n)$. By definition, 
$\varphi_i:= x_i^{-1}\varphi (x_i) \in \Gamma_{k+1}.$
The Jacobian matrix of $\varphi$ can be described explicitly:
\[(D\varphi)_{ij} =  \frac{\partial (x_i\varphi_i)}{\partial x_j}
                  =  \underbrace{\frac{\partial x_i}{\partial x_j}}_{\delta_{ij}} 
               + x_i \frac{\partial (\varphi_i)}{\partial x_j}. \]
Hence:
\[D\varphi - \Id{} = \left(x_i \frac{\partial \varphi_i}{\partial x_j} \right)_{ij}.\]
Using Remark \ref{derivation_on_LCD}, we see that this matrix is in fact in $M_n(I^k)$ (to shorten notations, we write $I$ for $I F_n$ in the sequel).
Moreover, $x_i$ acts trivially on $I^k/I^{k+1}$. We thus get an explicit formula for the trace map:
\[\tr (D\varphi - \Id{}) = \sum\limits_i x_i\frac{\partial \varphi_i}{\partial x_i} \equiv \sum\limits_i \frac{\partial \varphi_i}{\partial x_i} \pmod{I^{k+1}}.\]

\medskip

Let $G$ be any group. We can apply the construction of Paragraph \ref{congruence_filtration} to $A = \k G$, filtered by the powers of the augmentation ideal. This gives a strongly central filtration $GL_n(I^*G)$ on $GL_n(IG)$, which comes with an embedding of Lie algebras:
\[\Lie(GL_n(I^*G)) \hookrightarrow \gr(M_n(\k G)) \cong M_n(\gr(\k G)).\]

\smallskip

The next proposition replaces a formula from \cite[section 6]{Bartholdi}:
\begin{prop}\label{gr_de_D}
The Jacobian matrix $D$ induces a morphism between graded modules:
\[D: \Lie(\mathcal A_*(F_n)) \longrightarrow \Lie(GL_n(I^*F_n)),\]
satisfying:
\[D([f,g]) = \left[g, Df\right] + \left[Dg, f\right] + \left[Dg,Df\right].\]
\end{prop}

In order to address the issue raised in Remark \ref{sens_D(fg)}, let us introduce some notations before proving the proposition. If $G$ is any group, we denote by $G^{op}$ the \emph{opposite group}, where multiplication is defined by:
$g \cdot_{op} h = hg.$
Let $G_*$ be a strongly central filtration on $G$. Then $G_*^{op}$ is such a filtration on $G^{op}$ and one easily checks that:
\[\Lie(G_*^{op}) = \Lie(G_*)^{op},\]
where the bracket in $\Lie(G_*)^{op}$ is $\Lie(G_*)$'s additive inverse: $[x,y]_{op} = [y,x].$

\begin{proof}[Proof of Proposition \ref{gr_de_D}]

Corollary \ref{Dérivation_composée} states exactly that $D$ is a derivation from $\Aut(F_n)$ to $GL_n(I)^{op} \subset M_n(\Z F_n)^{op}$, where $M_n(\Z F_n)$ is endowed with the obvious $\Aut(F_n)$-action.
We thus can apply the results from \ref{derivations} with $\partial = D$, $G = \Aut(F_n)$, $H = GL_n(I)^{op}$, and $H_* =  GL_n(I^*)^{op}$. 

The strongly central filtration $\mathcal A_*(G,H_*)$ is in fact the Andreadakis filtration $\mathcal A_* = \mathcal A_*(F_n)$ on $\mathcal A_1 = IA_n \subset \Aut(F_n)$. Indeed, there is a series of inclusions:
\[\mathcal A_*\left(G, D_*F_n\right)        \supseteq
  \mathcal A_*\left(G, GL_n(I^*)\right) \supseteq
  \mathcal A_*\left(G, M_n(\Z F_n)\right) =
  \mathcal A_*\left(G, \Z F_n \right).\]
The first one comes from \ref{inj-surj} applied to the injection of $D_*F_n$ into $GL_n(I^*)$ defined by $w \longmapsto w \cdot \Id{}$. The second one is a particular case of \ref{incl_de_A}. The last equality comes from the fact that $G$ acts component-wise on matrices:
\[[g, (m_{ij})] = g \cdot (m_{ij}) - (m_{ij}) = ([g, m_{ij}]).\]
According to proposition \ref{égalité_de_A}, these inclusions are in fact equalities.

Moreover, we have seen at the beginning of the present paragraph that $D$ sends $\mathcal A_*$ to $GL_n(I^*)^{op}$. The filtration $\mathcal A_* \cap \partial^{-1}(H_*)$ is thus only $\mathcal A_*$. The work already done in \ref{derivations} allows us to get the desired result.
\end{proof}

The map given by Proposition \ref{gr_de_D} can be composed with the morphism:
\[\Lie(GL_n(I^*)) \overset{(-) - \Id{}}{\longrightarrow} \gr(M_n(\Z F_n)) \cong M_n(\gr(\Z F_n)) = M_n(TV).\]
Thus, for $\varphi$ an element of $\mathcal A_k/\mathcal A_{k+1}$, $D\varphi - \Id{}$ is well-defined modulo $M_n(I^{k+1})$.
Composing with the usual trace, we get the announced well-defined linear map induced by $\varphi \longmapsto \tr(D\varphi - \Id{})$:
\[\tr: \Lie(\mathcal A) \longrightarrow TV.\]

\begin{rmq}
That the map $D-\Id{}$ (hence $\tr$) induces a well-defined map between the Lie algebras can be seen through explicit calculation, but the behaviour with respect to the Lie bracket is much less obvious from this point of view.
Indeed, let $\varphi \in \mathcal A_k$. If $\varphi = \psi \chi$, with $\chi \in \mathcal A_{k+1}$, then:
\[\psi(x_i) = \varphi(x_i \chi_i) = x_i \varphi_i \varphi(\chi_i).\]
As $\chi_i$ stands inside $\Gamma_{k+2}$, its image by $\varphi$ does too. Thus:
\[\frac{\partial \psi_i}{\partial x_j} = 
\frac{\partial \varphi_i}{\partial x_j} + \varphi_i \underbrace{\frac{\partial \varphi(\chi_i)}{\partial x_j}}_{\in I^{k+1}} \equiv 
\frac{\partial \varphi_i}{\partial x_j} \pmod{I^{k+1}}. \]
\end{rmq}

\subsubsection{Introducing the contraction map}\label{contraction_map}

Consider the evaluation map:
\[ev: \Der_{(\Id{},\varepsilon)}(TV) \otimes TV \longrightarrow TV.\]
Using the universal property of $TV$, as in the proof of \ref{graded_of_derivations}, we get a linear isomorphism:
\[\Der_{(\Id{},\varepsilon)}(TV) \cong \Hom(V,TV) \cong V^* \otimes TV.\]
The evaluation map then is:
\[\left\{\begin{array}{lll}
         V^* \otimes TV \otimes TV  &\longrightarrow &TV \\
         \omega \otimes u \otimes v &\longmapsto     &\Phi(\omega \otimes v)u,
         \end{array}
\right.\]
where $\Phi$ is the contraction map:
\[\Phi: \left\{\begin{array}{lll}
                 V^* \otimes V^{\otimes k + 1} 
                 &\longrightarrow & V^{\otimes k}\\
				 \alpha \otimes X_{i_1} \cdots X_{i_{k+1}}              
				 &\longmapsto     & X_{i_1} \cdots X_{i_k} \alpha(X_{i_{k+1}}),
				        \end{array} \right.\]
extended by zero on $\k \cdot 1$. This follows from the fact that any $(\Id{}, \varepsilon)$-derivation $\partial$ verifies:
\[\partial(uv) = u \cdot \partial v,\]
when the degree of $v$ is at least $1$ (that is, when $\varepsilon(v) = 0$), and $\partial(1) = 0$.

\medskip

We sum this up in the following:
\begin{prop}
Let $\partial \in \Der_{(\Id{}, \varepsilon)}(TV)$. Then:
\[\partial = \Phi(\partial|_V \otimes -).\]
\end{prop}

Let us consider the derivation $\frac{\partial}{\partial x_i}$ of $\Z F_n$. It induces a $(\Id{}, \varepsilon)$-derivation of degree $-1$ of $TV$, denoted by $\partial_i$ (any derivation of $\k G$ is of degree at least $-1$, by Corollary \ref{degree_of_derivations_of_kG}). As $\partial_i|_V = X_i^*$, we get the following:
\begin{cor}\label{derivations_and_contractions_2} The $(\Id{}, \varepsilon)$-derivation of $TV$ induced by $\frac{\partial}{\partial x_i} \in \Der(\Z F_n)$ is represented as:
\[\overline{\frac{\partial}{\partial x_i}} = \partial_i = \Phi (X_i^* \otimes -): TV \longrightarrow TV.\]
\end{cor}

We can use these results to interpret the trace map in a way more suited to explicit calculations:
\begin{prop}\label{algebraic_description_of_tr}
The trace map can be described as:
\[\tr = \Phi \circ \iota \circ \tau,\]
where $\tau$ is the Johnson morphism (see Definition \ref{déf_tau}), $\iota$ denotes the inclusion of $\Der_k(LV) \cong V^* \otimes L_{k+1}V$ into $\Der_k(TV) \cong V^* \otimes V^{\otimes k+1}$, and $\Phi$ is the contraction map.
\end{prop}

\begin{proof}
Let $\varphi \in \mathcal A_k$. Then $\tau(\varphi)$ is defined by:
\[\tau(\varphi)(X_i) = \overline{x_i^{-1} \varphi(x_i)} = \overline{\varphi_i}\ \in \Gamma_{k+1} / \Gamma_{k+2} \cong \mathfrak L_{k+1}V.\]
We have seen at the beginning of Paragraph \ref{gradué} that the trace map is given by:
\[\tr (D \varphi - \Id{}) = \sum\limits_i \frac{\partial \varphi_i}{\partial x_i}.\]
The formula of the proposition is then equivalent to:
\[\Phi \left( \sum\limits_i X_i^* \otimes \overline{\varphi_i} \right) = \overline{\sum\limits_i \frac{\partial \varphi_i}{\partial x_i}}.\]
To get this formula, we evaluate the equality given by Corollary \ref{derivations_and_contractions_2} to the elements $\overline{\varphi_i - 1}$ (keeping in mind that the inclusion of  $L_{k+1}V$ into $T_{k+1}V$ is given by $\overline{w} \mapsto \overline{w-1}$).
\end{proof}

\subsection{Stable surjectivity}

\subsubsection{Vanishing of the trace map}

Here, we show that the trace map takes values in brackets inside $TV$. This result can also be found in \cite[Prop. 5.3]{Massuyeau2}, where rational methods are used to get it. 

\begin{prop}[{\cite[Th.\ 2.1]{Bryant}}, quoted in {\cite[Th.\ 6.2]{Bartholdi}}]\label{Tr_o_tau}
Let $k \geq 2$, and let $J \in GL_m(I^k_{\Z} F_n)$. Denote by $V$ the abelianization $V = F_n^{ab} \cong \mathbb \Z^n$. Then:
\[\tr \left( J - \Id{} \right) \in [TV,TV]_k \subset V^{\otimes k} \cong I^k/I^{k+1}.\]
\end{prop}

This result relies on the following criterion:

\begin{prop}[{\cite[prop. 2.2]{Bryant}}]\label{criterion_for_brackets}
Let $f(X_1, ..., X_n) \in V^{\otimes k}$. Let $\mathcal C \subset M_k(\Z)$ be the sub-$\Z$-module generated by the $\textbf{e}_{i,i+1}$. Suppose: 
\[\forall C_i \in \mathcal C,\ \tr(f(C_1, ..., C_n)) = 0.\]
Then $f \in [TV,TV]_k$.
\end{prop}

The proof can be found in \cite{Bryant}. The reader is also referred to the proof of Proposition \ref{criterion_for_brackets-p}, which is the same proof, adapted to the case of positive carateristic.

\begin{proof}[Proof of Proposition \ref{Tr_o_tau}]
The main idea is to use evaluations into commutative algebras to be able to use Proposition \ref{det_and_tr}, and to then get back to the non-commutative setting by using the above criterion.

\smallskip

Let $\Id{} + tA_i \in GL_k(t\k[t])$. There is an evaluation morphism $x_i \mapsto \Id{} + tA_i$ from $F_n$ to $GL_k(t\Z[t])$, extending to a morphism from $\k F_n$ to $M_k(t\k[t])$ sending $I^*$ to $t^* M_k(\k[t])$. Taking congruence groups, we get an evaluation morphism:
\[ev_{\Id{} + tA_i}: GL_m(I^* F_n) \longrightarrow  GL_m(t^* M_k(\k[t])) = GL_{mk}(t^* \k[t]). \]
There is a commutative diagram:
\[\begin{tikzcd}
\Lie(GL_m(I^* F_n)) \ar[r, hook, "(-)- \Id{}"] \ar[d, "ev_{\Id{} + tA_i}"]
&M_m(TV) \ar[r, "Tr"] \ar[d, "ev_{tA_i}"]
&TV \ar[d, "ev_{tA_i}"]\\
\Lie(GL_{mk}(t^* \k[t])) \ar[r, hook, "(-)- \Id{}"] \ar[d, "det"]
&M_{mk}(\k [t]) \ar[r, "Tr_m"] \ar[d, "Tr"] 
&M_m(\k[t]) \ar[ld, "Tr"] \\
\Lie(\k[t]^\times_*)=0 \ar[r, hook, "(-)- 1"]
&\k[t]
\end{tikzcd}\]
Here, we identify $\gr(I^*F_n)$ with $TV$ by $x_i \mapsto 1+X_i$. We also identify $\gr(t^*\k[t])$ with $\k[t]$. The evaluation $x_i \mapsto \Id{} + tA_i$ thus induces $X_i \mapsto tA_i$ between the associated graded. The bottom-left square is just the one in Proposition \ref{det_and_tr}. The map $Tr_m$ is the usual trace when the base algebra is $M_m(\k[t])$.

\smallskip

Let $\k = \Z$ and $f = \tr(J - \Id{}) \in V^{\otimes k}$. The commutativity of the above diagram gives:
\[0 = \tr(f(tA_i)) = t^k\tr(f(A_i)).\]
As a consequence, $\tr(f(A_i)) = 0$, for any $\Id{} + tA_i \in GL_k(t\Z[t])$. We can then evaluate this at $t = 0$ to get: $\tr(f(\pi A_i)) = 0$. This evaluation $\pi$ is the map:
\[\pi: \Lie_1(GL_k(t^*\Z[t]) \hookrightarrow M_k \left(t\Z[t]/t^2\Z[t] \right) \cong M_k(\Z).\]
Using Proposition \ref{gr_of_sl_n} and Remark \ref{SL_n_et_GL_n}, we see that its image is exactly $\mathfrak{sl}_n(\Z)$, so the conclusion follows from the above criterion (\ref{criterion_for_brackets}).
\end{proof}

Because of Proposition \ref{Tr_o_tau}, we will consider the trace map as taking values in the abelianization $TV^{ab} = TV/[TV,TV]$. As $[TV,TV]_k$ is generated by the elements:
\[[X_{i_1} \otimes \cdots \otimes X_{i_p}, X_{i_{p+1}} \otimes \cdots \otimes X_{i_k}] = X_{i_1} \otimes \cdot \otimes X_{i_k} - t^p \cdot X_{i_1} \otimes \cdot \otimes X_{i_k},\]
where $t = \bar 1 \in \Z /k$, the module $TV^{ab}$ is the module of \emph{cyclic powers} $C_*V$:
\[(TV^{ab})_k = C_k V:= V^{\otimes k}/(\Z /k).\] 
The conclusion of Proposition \ref{Tr_o_tau} becomes, in this context: $\tr \left( J - \Id{} \right) = 0 \in C_*V.$

\subsubsection{Linear algebra}\label{par-linear_alg}

Consider the Johnson morphism 
$\tau': \Lie \left(\Gamma_{IA_n} \right) \longrightarrow \Der_*\left(\Lie(F_n) \right)$
(cf. \ref{déf_tau'}). 
The morphism $\tau'_1$ is an isomorphism (see Proposition \ref{IAn^ab}). Moreover, $\Lie \left(\Gamma_{IA_n} \right)$ is generated in degree $1$ (cf. \ref{engdeg1}). As a consequence, the image of $\tau'$ is exactly the sub-Lie ring generated in degree $1$ inside $\Der \left(\Lie(F_n) \right)$. As $\Lie(F_n)$ is the free Lie ring $\mathfrak LV$, the study of $\coker(\tau')$ is solely a problem of linear algebra.

\medskip

Recall from Proposition \ref{algebraic_description_of_tr} that the trace map can be seen as the composite of the Johnson morphism $\tau: \Lie_k(\mathcal A_*(F_n)) \rightarrow \Der_k(\mathfrak LV) \cong V^* \otimes \mathfrak L_{k + 1}V$ with:
\[\tr_M: V^* \otimes \mathfrak L_{k + 1}V \overset{\iota}{\longrightarrow} 
         V^* \otimes V^{\otimes k + 1}    \overset{\Phi}{\longrightarrow}
         V^{\otimes k}                    \overset{\pi}{\longrightarrow}
         C_kV:= V^{\otimes k}/(\Z /(k)),  \]
where $\iota$ and $\pi$ denote the canonical maps. All these morphisms are obviously $\GL_n(\Z)$-equivariant (with respect to the canonical actions).

\begin{nota}
Let $\mathfrak I$ denote the image of $\tau'$, which is the sub-Lie ring generated in degree $1$ inside $\Der(\mathfrak LV)$.
\end{nota}

The following proposition can be seen as a consequence of Proposition \ref{Tr_o_tau}. Precisely, $\mathfrak I = \ima(\tau') \subseteq \ima(\tau)$, and Proposition \ref{Tr_o_tau} implies that $\tr_M \circ \tau = \tr$ vanishes.
\begin{prop}\label{Tr(I)}
For every $k \geq 2$, $\tr_M(\mathfrak I_k) = \{0\}.$
\end{prop}

\subsubsection{Stable cokernel of \texorpdfstring{$\tau'$}{t'} and stable surjectivity}\label{coker}

Let $k \geq 2$. Using Proposition \ref{Tr(I)}, we get a commutative diagram with exact rows:
\[\begin{tikzcd}
\mathfrak I_k                    \ar[r, hook] \ar[d, dashed, "\phi"]
& V^* \otimes \mathfrak L_{k+1}V \ar[r, two heads]     \ar[d, "\Phi"]
& \coker(\tau'_k)                \ar[d, dashed , "\overline{\Phi}"] \\
{[TV, TV]_k}      \ar[r, hook]   
& V^{\otimes k}   \ar[r, two heads, "\pi"]               
& C_k V.
\end{tikzcd}\]
            
In \cite{Satoh2}, Satoh shows:
\[
\left\{
\begin{array}{lll}
  \text{For } n \geq k+1, &\Phi \text{ is surjective,}    & (\text{lemma } 3.2)\\
  \text{For } n \geq k+2, &\phi \text{ is surjective,}    & (\text{prop. } 3.2)\\
  \text{For } n \geq k+2, &\ker \Phi \subseteq \mathfrak I.& (\text{Prop.\ } 3.3)
\end{array}
\right.
\]
  
\medskip
  
His Theorem 3.1 is still true over $\Z$:

\begin{prop}\label{prop-coker} 
Let $k \geq 2$ and $n \geq k+2$ be integers. Then $\overline{\Phi}$ is a $GL_n(\Z)$-equivariant isomorphism:
\[\coker (\tau'_k) \cong C_kV.\]
\end{prop}

\begin{proof}
Let us denote by $K$ (resp.\ $L$) the kernel of $\overline{\Phi}$ (resp.\ its cokernel). There is a commutative diagram in $GL_n(\Z)-\kMod[\Z]$:

\[\begin{tikzcd}
&\ker \Phi \ar[r, dashed, "0"] \ar[d] \ar[ld, dashed]
&K                        \ar[d] \\
\mathfrak I_k  \                  \ar[r, hook] \ar[d, dashed, "\phi"]
& V^* \otimes \mathfrak L_{k+1}V \ar[r, two heads]     \ar[d, "\Phi"]
& \coker(\tau')                                  \ar[d, dashed, "\overline{\Phi}"]  \\
{[TV, TV]_k}        \ar[r, hook] \ar[d]   
& V^{\otimes k}   \ar[r, two heads, "\pi"] \ar[d]             
& C_k V                           \ar[d] \\
0
&0 
&L. 
\end{tikzcd}\]

The snake lemma ensures that $K$ and $L$ are zero: $\overline{\Phi}$ is an isomorphism.
\end{proof}

\medskip

We can now state our main result:
\begin{theo}\label{stable_surj}
Let $k + 2 \leq n$. Then the canonical morphism
\[\Lie_k (IA_n) \longrightarrow \Lie_k (\mathcal A_*(F_n))\]
is surjective, and $\tau$ induces an isomorphism:
$\Lie_k (\mathcal A_*(F_n)) \cong \mathfrak I_k.$
\end{theo}

\begin{rmq}
Basis being chosen, there is an injection of $F_n$ in $F_{n+1} \cong F_n * \Z$. An automorphism $\varphi$ of $F_n$ can be extended to an automorphism $\varphi * \Id{}$ of $F_{n+1}$. This induces injections $IA_n \hookrightarrow IA_{n+1}$which in turn induce morphisms $\Lie (\mathcal A_*(F_n)) \rightarrow \Lie (\mathcal A_*(F_{n+1}))$. Taking the colimit over $n$, we can define a Lie ring $\Lie^{st}(\mathcal A_*)$. In the same way, we can define injections from $\Der(\mathcal L(F_n^{ab}))$ into $\Der(\mathcal L(F_{n+1}^{ab}))$ and take the colimit $\mathfrak I^{st}$ of the sub-algebras generated in degree one.
With this point of view, the isomorphisms of Theorem \ref{stable_surj} give an isomorphism between graded Lie algebras: 
\[\tau^{st}: \Lie^{st}(\mathcal A_*) \cong \mathfrak I^{st},\]
meaning exactly that $\Lie^{st}(\mathcal A)$ is generated in degree one.

In fact, all the constructions appearing here are functors on the category denoted by $\textbf{S}(\Z)$ in \cite[section 7]{Djament-finitude}, where it is shown (using methods similar to the ones of \cite{Church-noeth}) that these functors are finitely supported. This implies the equivalence between $\tau_k^{st}$ being an isomorphism and $\tau_k$ being one for $n$ big enough.
\end{rmq}

\begin{proof}[Proof of theorem \ref{stable_surj}]
Consider the commutative diagram:
\[\begin{tikzcd}
\Lie(IA_n) \ar[r, "i_*"] \ar[rd, swap, "\tau'"] 
&\Lie(\mathcal A_*) \ar[d, hook, "\tau"] \\
&\Der(\mathfrak LV).
\end{tikzcd}\]
The image $\mathfrak I$ of $\tau'$ is the sub-Lie ring generated in degree one inside $\Der(\mathfrak LV)$. Using the results quoted in Paragraph $\ref{coker}$, we see that in degrees $k \leq n-2$, it also is the kernel of the trace map.
Proposition (\ref{Tr_o_tau}) tells us exactly that $\tr \circ \tau = 0$, so that $\ima \tau_k \subseteq \ker \tr_k = \ima \tau'_k$ when $k \geq n-2$. As a consequence, $\ima \tau = \ima \tau'$. As $\tau$ is injective (\ref{tau_inj}), it is an isomorphism onto its image, hence the result.
\end{proof}

\subsection{Automorphisms of free nilpotent groups}

Automorphisms of nilpotent groups are easy to deal with, due to the following classical fact:

\begin{lem}\label{iso_test_on_ab}
Let $G$ be a finite-type nilpotent group. An endomorphism $\varphi \in \End(G)$ is an automorphism if and only if the induced morphism $\varphi^{ab} \in \End(G^{ab})$ is.
\end{lem}

\begin{proof}
If $\varphi$ is an automorphism then $\varphi^{ab}$ has to be, with $(\varphi^{ab})^{-1} = (\varphi^{-1})^{ab}$. Conversely, suppose that $\varphi^{ab}$ is an automorphism. This means that $\Lie_1(\varphi)$ is. Since $\Lie(G)$ is generated in degree one, $\Lie(\varphi)$ is surjective. But each $\Lie_k(G)$ is abelian of finite type, so each $\Lie_k(\varphi)$, being surjective, has to be bijective (it is obviously the case on the torsion part, which is finite, and it also is on the free abelian part, for reasons of rank). The lemma then follows by induction from the five-lemma applied to:
\[\begin{tikzcd}
\Lie_k(G) \ar[r, hook] \ar[d, "\Lie_k(\varphi)"]
& G/\Gamma_{k+1}G \ar[r, two heads] \ar[d, "\bar\varphi"]
& G/\Gamma_k G \ar[d, "\bar\varphi"] \\
\Lie_k(G) \ar[r, hook] 
& G/\Gamma_{k+1}G \ar[r, two heads]
& G/\Gamma_k G. 
\end{tikzcd}\] 
This induction process stops since there is a $c$ such that $G = G/\Gamma_{c+1}G$.
\end{proof}

\begin{defi}
The \emph{pro-nilpotent completion} of a group $G$ is:
$\widehat G := \varprojlim \left(G/\Gamma_k G \right).$
\end{defi}
The completion $\widehat G$ is canonically filtered by the $\bar\Gamma_j G := \varprojlim \left(\Gamma_j G/\Gamma_k G \right)$. This filtration is its \emph{closed lower central series}, defined as the closure of the lower central series. It is minimal amongst \emph{closed} strongly central filtrations on $\widehat G$.
An endomorphism of $\widehat G$ is continuous if and only if it preserve this filtration.
\begin{lem}\label{iso_test_on_ab-c}
Let $G$ be a group. A continuous endomorphism $\varphi$ of $\widehat G$ is an automorphism if and only if the induced morphism $\varphi^{ab} \in \End(G^{ab})$ is.
\end{lem}
\begin{proof}
Such an endomorphism is an automorphism if and only if the associated morphism between projective system is. These are the induced endomorphisms of the $G/\Gamma_k G$, which are nilpotent groups. This condition amounts to $\varphi$ inducing an isomorphism on $\hat G/\bar\Gamma_2 = G^{ab}$, by Lemma \ref{iso_test_on_ab}.
\end{proof}

In fact, we can readily deduce the following explicit description of the group $\Aut_{\mathcal C^0}(\widehat G)$ of continuous automorphisms of $\widehat G$:
\begin{prop}\label{proj_lim_of_Aut}
The canonical map is an isomorphism:
\[\Aut_{\mathcal C^0}(\widehat G) \cong \varprojlim \left(\Aut(G/\Gamma_k G)\right).\]
\end{prop}

\medskip

Let $G = F_n$ be a free group of finite type. It is residually nilpotent, so it embeds into its completion $\widehat F_n$. A continuous endomorphism of $\widehat F_n$ is uniquely determined by its (arbitrary) values on the topological generators $x_i$. 

Let us denote by $F_{n,c}$ the free $c$-nilpotent group $F_n/\Gamma_{c+1}F_n$. Using Lemmas \ref{iso_test_on_ab} and \ref{iso_test_on_ab-c}, we can show two surjectivity results for automorphisms of free nilpotent groups:

\begin{prop}\label{surj_between_Aut}
The canonical morphisms $\Aut(F_{n,c}) \rightarrow \Aut(F_{n,c-1})$ are surjective.
\end{prop}
\begin{proof}
Let $(x_i)$ be a free basis of $F_n$. Let $\varphi \in \Aut(F_n/\Gamma_c(F_n))$. Lift the elements $\varphi(\bar x_i)$ to elements $t_i$ of $F_n/\Gamma_{c+1}(F_n)$. Then define the endomorphism $\tilde\varphi$ of $F_n/\Gamma_{c+1}(F_n)$ by  $\bar x_i \mapsto t_i$. Since $\varphi$ and $\tilde\varphi$ induce the same endomorphism of $F_n^{ab}$, Lemma \ref{iso_test_on_ab} implies that $\tilde\varphi$ is an isomorphism.
\end{proof}

\begin{prop}[{\cite[th.\ 5.1]{Bartholdi}}]
The Johnson morphism associated to the universal action on $\bar\Gamma_*(F_n)$ is an isomorphism:
\[\tau: \Lie(\mathcal A_*(\bar\Gamma_*F_n)) \cong \Der(LV).\]
\end{prop}
\begin{proof}
Because of Lemma \ref{Johnson_inj}, we know that $\tau$ is injective. We need to show that it is surjective.
Let us first remark that $\Lie(\bar\Gamma_*(F_n)) = \Lie(\Gamma_*(F_n)) \cong LV$. Let $\partial \in \Der_k(LV)$. Lift each $\partial(X_i) \in \Gamma_k/\Gamma_{k+1} \cong \bar\Gamma_k/\bar\Gamma_{k+1}$ to an element $t_i \in \bar\Gamma_k$. We can define a continuous endomorphism of $\widehat F_n$ by $\varphi : x_i \mapsto t_i x_i$. Then $\varphi$ acts trivially on $F_n^{ab}$, so it is an isomorphism by Lemma \ref{iso_test_on_ab-c}. As a consequence, $\varphi \in \mathcal A_k(\bar\Gamma_*(F_n))$ satisfies $\tau(\bar \varphi) = \partial$. This concludes the proof.
\end{proof}

We can translate this last result into a statement about automorphisms of free nilpotent groups : because of Propositions \ref{proj_lim_of_Aut} and \ref{surj_between_Aut}, $\Aut(F_{n,c})$ is a quotient of $\Aut_{\mathcal C^0}(\widehat F_n)$. Moreover, the kernel of the canonical surjection is $\mathcal A_c(\bar\Gamma_*F_n)$, by definition. Thus, this projection induces the $c$-truncations at the level of the associated graded objects :
\begin{cor}\label{Johnson_iso}
The Johnson morphism associated to the universal action on $\Gamma_*(F_{n,c})$ is an isomorphism:
\[\tau: \Lie(\mathcal A_*(F_{n,c})) \cong \Der(L_{\leq c}V).\]
\end{cor}

Consider the Johnson morphism $\tau': \Lie(\Gamma_*(IA_{F_{n,c}})) \rightarrow \Der(L_{\leq c}V)$. Using the identification of corollary \ref{Johnson_iso}, we see that $\tau'$ identifies with the Andreadakis morphism $i_*: \Lie(\Gamma_*(IA_{F_{n,c}})) \rightarrow \Lie(\mathcal A_*(F_{n,c}))$. 
Its image is the subalgebra generated in degree one inside $\Der(L_{\leq c}V)$, as was the case in paragraph \ref{par-linear_alg}. This subalgebra is exactly the truncation $\mathfrak I_{<c}$, so is inside (and stably equal to) the kernel of the trace map. As a consequence, \textbf{the Andreadakis equality never holds for free nilpotent groups}. Moreover, in this context, our stable surjectivity result translates as :
\begin{cor}
The following sequence always is a complex, and is exact for  $n \geq c + 1$:
\[\Lie(\Gamma_*(IA_{F_{n,c}})) \overset{i_*}{\longrightarrow} \Lie(\mathcal A_*(F_{n,c})) \cong \Der(L_{\leq c}V) \overset{\tr}{\longrightarrow} C_{<c}V \longrightarrow 0.\]
\end{cor}

\begin{rmq}[Non-tame automorphisms]
The canonical morphism $p: \Aut(F_n) \rightarrow \Aut(F_{n,c})$ is in general not surjective: some basis of the free nilpotent group do not lift to basis of the free group \emph{via} $F_n \twoheadrightarrow F_{n,c}$. Automorphisms of $F_{n,c}$ induced by automorphisms of $F_n$ are called \emph{tame}. This was the original motivation of \cite{Bryant} for considering the trace map. In this regard, our stable surjectivity result could be re-stated as follows : in the stable range, the trace is the only obstruction for an automorphism to be tame.
\end{rmq}

\section{The case of positive caracteristic}\label{Section-p}

\subsection{Dark's theorem}

Let $w$ be a word in a free group $F_S$. Let $w(\textbf{r})$ be the word obtained from $w$ by replacing each generator $s \in S$ by some power $s^{r_s}$. Dark's theorem \cite{Dark} describes how to decompose $w(\textbf{r})$ as a product of commutators. This gives very useful universal formulas, that can then be evaluated in any group. 
The reader is referred to \cite[chap. IV, Th.\ 1.11]{Passi} for a precise statement and a proof of the theorem. Here we recall two corollaries, obtained by taking $w = [x,y]$ and $w = xy$ in $F_{\{x,y\}}$.

\medskip

The case $w = [x,y]$ is \cite[IV, cor. 1.16]{Passi}:
\begin{cor}\label{Dark_for_commutators}
There exists a unique map $\theta: (\mathbb N^*)^2 \longrightarrow F_2 = \langle x,y \rangle$ verifying:
\[\forall \alpha, \beta \in \mathbb N,\ [x^\alpha,y^\beta] = \prod\limits_{r,s \geq 1} \theta(r,s)^{\binom{\alpha}{r}\binom{\beta}{s}}.\]
Each $\theta(r,s)$ is a product of $\{x^{\pm 1},y^{\pm 1}\}$-commutators such that $x^{\pm 1}$ appears at least $r$ times and $y^{\pm 1}$ at least $s$ times in each factor.
\end{cor}

The case $w = xy$ has been known for a long time (quoted in \cite[chap. 11, Th.\ 1.14]{Passman}, it already appears for instance in \cite{Hall}):% [PRECISE REF ?]

\begin{cor}\label{Dark_for_products}
There exists a unique map $\theta: \mathbb N \longrightarrow F_2 = \langle x,y \rangle$ such that:
\[\forall \alpha \in \mathbb N,\ x^\alpha y^\alpha = \prod\limits_{r \geq 0} \theta(r)^{\binom{\alpha}{r}},\]
Each $\theta(r)$ is a product of $\{x^{\pm 1},y^{\pm 1}\}$-commutators of length at least $r$.
\end{cor}

\begin{rmq}
In both cases, uniqueness of the map $\theta$ is obvious: it can be defined by induction.
\end{rmq}

\subsection{\texorpdfstring{$p$}{p}-Restricted strongly central series}

\begin{defi}
Let $p$ be a prime number. A strongly central series $G_*$ is said to be \emph{$p$-restricted} if:
\[\forall i,\ G_i^p \subseteq G_{ip}.\]
\end{defi}

Let $G_*$ be a $p$-restricted strongly central series. Using \cite[Th.\ III.1.7]{Passi}, we see that the morphism:
\[\Lie(G_*) \longrightarrow \gr(\mathfrak a^{\F_p}_*(G_*))\]
induced by $g \longmapsto g-1$ (see Theorem \ref{Lazard} and Proposition \ref{adjonction_a}) is injective. We can identify $\Lie(G_*)$ with its image, which is stable by the $p$-th power operation in the associative $\mathcal F_p$-algebra $\gr(\mathfrak a^{\F_p}_*(G_*))$, since $(g-1)^p = g^p-1$. From this we deduce that $\Lie(G_*)$ is a \emph{$p$-restricted Lie algebra} with $p$-th power operation induced by $g \longmapsto g^p$ in $G_1$. The reader is referred to the classical \cite{Jacobson} for a discussion of $p$-restricted Lie algebras.
We also can deduce from \cite[Th.\ III.1.7]{Passi} that any $p$-restricted strongly central series has to contain the dimension series $D^{\F_p}_*G$ defined in Example \ref{dimension} (because $\mathfrak a^{\F_p}_*(G_*)$ contains $I^*_{\F_p}G$):
\begin{prop}
The filtration $D^{\F_p}_*G$ is the minimal $p$-restricted strongly central series, on any group $G$.
\end{prop}
The filtration $D^{\F_p}$, also denoted by $\Gamma_*^{[p]}$ also admits a description by induction, or the more explicit description:
\[\Gamma_k^{[p]}G = \prod_{ip^j \geq n}(\Gamma_iG)^{p^j}.\]
These can be found in \cite[Th.\ 5.6]{Lazard} or \cite[Th.\ IV.1.9]{Passi}. We also refer to \cite{Chapman-Efrat} for a nice discussion of filtrations defined by induction. This description has a nice consequence, similar to Proposition \ref{engdeg1} (we abbreviate $\Lie(\Gamma_*^{[p]}G)$ to $\Lie^{[p]}(G)$):
\begin{prop}\label{engdeg1-p} 
The $p$-restricted Lie ring $\Lie^{[p]}(G)$ is \emph{generated in degree $1$}. Precisely, it is generated (as a $p$-restricted Lie ring) by $\Lie_1^{[p]}(G) = G^{ab} \otimes \F_p$.
\end{prop}

\begin{ex}\label{Lp(Fn)_free}
If $G$ is a free group, then $\Lie^{[p]}(G) = \Lie(D^{\F_p}_*G)$ is the $p$-restricted Lie algebra generated by the degree-one part inside $\gr(\F_p G) \cong T_{\F_p}(G^{ab})$ so, using PBW over $\F_p$ \cite [Th.\ 1]{Jacobson} it is the free $p$-restricted Lie algebra over the $\F_p$-module $G^{ab} \otimes \F_p$ \cite[Th.\ 6.5]{Lazard}.
\end{ex}

\begin{rmq}
The filtration $\Gamma_*^{[p]}G$ (already defined in \cite{Zassenhaus}) is to be distinguished from \emph{Stallings' filtration} $\Gamma_*^{(p)}G$, defined in \cite{Stallings}. The latter is the minimal $p$-torsion strongly central filtration on a group $G$, where a strongly central filtration is \emph{$p$-torsion} when $G_i^p \subseteq G_{i+1}$ (for all $i$). See remark \ref{rk_on_q-torsion} for more on $q$-torsion strongly central series.
\end{rmq}

\subsection{The \texorpdfstring{$p$}{p}-restricted Andreadakis problem}\label{Andreadakis_p-restreint}

Let us denote $\mathcal A_*(\Gamma_*^{[p]}G)$ by $\mathcal A_*^{[p]}(G)$. Remark that $\mathcal A_1^{[p]}(G)$ is the group $IA^{[p]}_G$ of automorphisms acting trivially on $\Lie_1(\Gamma_*^{[p]}G) = G^{ab}\otimes \F_p$, hence on all of $\Lie(\Gamma_*^{[p]}G)$ (because of Proposition \ref{engdeg1-p}). If we show that $\mathcal A_*^{[p]}(G)$ is $p$-restricted (and we will -- see Proposition \ref{A p-restreinte}) then we get an inclusion:
\[\Gamma_*^{[p]}\left(IA^{[p]}_G\right) \subseteq \mathcal A_*^{[p]}(G).\]
We are thus led to consider a $p$-restricted version of the Andreadakis problem:
\begin{reppb}{pb_Andreadakis-p}[Andreadakis -- $p$-restricted version]
What is the difference between the $p$-restricted strongly central series $\mathcal A_*^{[p]}(G)$ and $\Gamma_*^{[p]}(IA^{[p]}_G)$ ? 
\end{reppb}

\begin{rmq}\label{IA^p_and_GL_n}
The group $IA^{[p]}_G$ contains $IA_G$ as a normal subgroup. Moreover, the quotient $IA^{[p]}_G/IA_G$ is a subgroup of $\Aut(G)/IA_G$, hence of $GL(G^{ab})$. In fact, by definition of $IA^{[p]}_G$, it is contained in the congruence group:
\[GL(pG^{ab}) = \ker(GL(G^{ab}) \rightarrow GL(G^{ab}\otimes \F_p)).\]
When $G = F_n$ is a free group of finite type, then $\Aut(G)/IA_G \cong GL_n(\Z)$, and $IA^{[p]}_G/IA_G$ is exactly $GL_n(p\Z)$.
\end{rmq}

\begin{prop}[\cite{Massuyeau1}, Prop.\ 8.5]\label{A p-restreinte}
Let $G_*$ be a $p$-restricted strongly central filtration, and $K$ be a group acting on $G_*$. Then $\mathcal A_*(K,G_*)$ is a $p$-restricted strongly central filtration.
\end{prop}

\begin{proof}
Let $\kappa \in \mathcal A_j(K,G_*)$ and $g \in G_i$. Using Corollary \ref{Dark_for_commutators}, we get:
\[[\kappa^p, g] = \prod\limits_{k=1}^p \theta(k,1)^{\binom{p}{k}},\]
with $\theta(k,1) \in G_{i+kj}$, for all $k$. If $k < p$, then
$\theta(k,1)^{\binom{p}{k}} \in G_{p(i+kj)} \subseteq G_{i+pj+1}.$
As $\theta(p,1)$ also is in $G_{i+pj}$, we have:
\[[\kappa^p, g] \in G_{i+pj}.\]
This is true for every $g \in G_i$, for all $i$. Hence $\kappa^p \in \mathcal A_{pj}(K,G_*)$, which completes the proof.
\end{proof}

Proposition \ref{A p-restreinte} can be refined:
\begin{prop}\label{A*N p-restricted}
Under the same hypothesis as Proposition \ref{A p-restreinte}, $\mathcal A_*(K,G_*) \ltimes G_*$ is a $p$-restricted strongly central series.
\end{prop}

\begin{proof}
Denote $\mathcal A_*(K,G_*) \ltimes G_*$ by $\mathcal K_*$. An element of $\mathcal K_j = \mathcal A_j \ltimes G_j$ is a product $\kappa \cdot g$, with $\kappa \in \mathcal A_j$ and $g \in G_j$. Using Corollary \ref{Dark_for_products}, we get:
\begin{equation}\label{eqn_Dark}
\kappa^p g^p = \prod\limits_{k=1}^p \theta(k)^{\binom{p}{k}} = (\kappa g)^p \cdot \theta(2)^{\binom{p}{2}} \cdots \theta(p-1)^p \cdot \theta(p),
\end{equation}
with $\theta(1) = \kappa g$ and $\theta(k) \in \mathcal K_{kj}$ for any $k$, as $\mathcal K_*$ is strongly central.
We use this formula to show, by induction on $d \leq p$, the following result: 
\[\forall j,\ \mathcal K_j^p \subseteq \mathcal K_{dj}.\]
This is true for $d = 1$, obviously. Let us assume that it holds for $d-1$.
Let $\kappa g \in \mathcal K_j$. For the sake of clarity, let us rewrite the formula \eqref{eqn_Dark}:
\[(\kappa g)^p = \kappa^p g^p \cdot \theta(p)^{-1} \cdot \prod\limits_{k=p-1}^2 \theta(k)^{-\binom{p}{k}}.\]
Using, respectively, that $\mathcal A_*$ is $p$-restricted (Proposition \ref{A p-restreinte}), that $G_*$ is (by definition) and that $\mathcal K_* = \mathcal A_* \ltimes G_*$ is strongly central, we get: 
\[\kappa^p,\ g^p,\ \theta(p)\ \in \mathcal K_{pj} \subseteq \mathcal K_{dj},\]
where the inclusion comes from the inequality $d \leq p$.
If $2 \leq k < p$, then $\theta(k) \in \mathcal K_{kj}$. As $p$ divides $\binom{p}{k}$, the induction hypothesis implies:
\[\theta(k)^{\binom{p}{k}} \in \mathcal K_{kj}^p \subseteq \mathcal K_{(d-1)kj} \subseteq \mathcal K_{dj},\]
because $(d-1)kj \geq dj$. 
Finally, we get what we were looking for:
\[(\kappa g)^p \in \mathcal K_{dj},\]
which completes the induction step, and the proof of the proposition.
\end{proof}

Let $\mathcal{SCF}_p$ be the full subcategory of $\mathcal{SCF}$ given by $p$-restricted strongly central series. As a consequence of Propositions \ref{A p-restreinte} and \ref{A*N p-restricted}, we get:
\begin{cor}
The category $\mathcal{SCF}_p$ is action-representative, the universal action on $G_*$ being $\mathcal A_*(G_*) \circlearrowright G_*$.
\end{cor}

This allows us to answer \cite[rk. 8.6]{Massuyeau1}. Indeed, the Lie functor restricts to a functor $\Lie: \mathcal{SCF}_p \longrightarrow p \mathcal Lie$ with values in the category $p \mathcal Lie$ of $p$-restricted Lie algebras (over $\F_p$). Actions in $p \mathcal Lie$ are represented by \emph{$p$-restricted derivations}, in the sense of Jacobson \cite{Jacobson}. As in Paragraph \ref{Johnson_section}, an action $K_* \circlearrowright G_*$ in $\mathcal{SCF}_p $ induces, by exactness of the Lie functor, an action $\Lie(K_*) \circlearrowright \Lie(G_*)$  in $p \mathcal Lie$, which is encoded by a morphism between $p$-restricted Lie algebras:
\[\tau: \Lie(K_*) \longrightarrow \Der^{[p]}(\Lie(G_*)),\]
where $\Der^{[p]} \subseteq \Der$ is the $p$-restricted sub-algebra of $p$-restricted derivations, \emph{i.e.\ }derivations $\partial$ satisfying:
\[\partial(a^p) = \ad_a^{p-1}(\partial a).\]
Let us stress that for $\mathfrak g \in p \mathcal Lie$, the Lie algebra $\Der(\mathfrak g)$ is indeed a $p$-restricted sub-algebra of $\End_{\F_p}(\mathfrak g)$, but is does not act on $\mathfrak g$ in $p \mathcal Lie$: the Lie algebra $\mathfrak g \rtimes\Der(\mathfrak g)$ bears no natural $p$-restricted structure.

\begin{rmq}
Using Proposition \ref{engdeg1-p} instead of \ref{engdeg1}, and replacing derivations by $p$-restricted ones in the proof, we can get an analogous of Lemma \ref{Explicit_A_*} for $\mathcal A_*^{[p]}(G)$: 
\[\mathcal A_j^{[p]}(G) = \Enstq{\sigma \in \Aut(G)}{[\sigma, G] \subseteq \Gamma_{j+1}^{[p]}(G)}.\]
In other words $\mathcal A_*^{[p]}(G)$ in the subgroup of automorphisms acting trivially on $G/\Gamma_{j+1}^{[p]}(G)$. This is exactly the definition used by Cooper \cite[def 3.2]{Cooper}.
\end{rmq}

\begin{rmq}[On the $q$-torsion case]\label{rk_on_q-torsion}
The same statements are true when considering $q$-torsion strongly central filtrations ($q$ does not have to be a prime number here), except that they are easier to show, because the condition $G_i^q \subseteq G_{i+1}$ is equivalent to the fact that $\Lie(G_*)$ is $q$-torsion. Precisely, if $\Lie(G_*)$ is $q$-torsion, then $\Der(\Lie(G_*))$ is too, and the injectivity of the Johnson morphism $\Lie(\mathcal A_*(G_*)) \hookrightarrow \Der(\Lie(G_*))$ implies that $\Lie(\mathcal A_*(G_*))$ also is. Hence $\mathcal A_*(G_*)$, and $\mathcal A_*(G_*) \ltimes G_*$ are $q$-torsion, so that these give a universal action on $G_*$ in the category of $q$-torsion  strongly central series. 

Moreover, $\Lie(G_*)$ also gets some kind of $q$-th power operation, induced by $q$-th powers in $G = G_1$. If $G_* = \Gamma_*^{(q)}G$ is Stallings' filtration on $G$, then these operations, together with the Lie algebra structure, generate $\Lie(G_*)$ from its degree one part, which allow us to get  an analogue of Lemma \ref{Explicit_A_*}: $\mathcal A_j(\Gamma_*^{(q)}G)$ is the subgroup of automorphisms acting trivially on $G/\Gamma_{j+1}^{(p)}(G)$. Using only this definition, Cooper managed to get the above results on the $q$-torsion case \cite{Cooper}. However, his claim that the $p$-restricted case worked similarly \cite[Lem.\ 3.7]{Cooper} seems flawed, and we do not see how to get it without the technical work done above (Proposition \ref{A p-restreinte}).

The minimality of Stallings' filtration also gives an inclusion:
\[\Gamma_*^{(q)}\left(IA^{[q]}_G \right) \subseteq \mathcal A_* \left(\Gamma_*^{(q)}G \right),\]
and a corresponding Andreadakis problem. Nevertheless, our methods in studying the Andreadakis problems so far rely heavily on algebraic structures associated to the dimension subgroups $D_*^{\mathbb Z}(F_n) = \Gamma_*(F_n)$ and $D_*^{\F_p}(F_n) = \Gamma_*^{[p]}(F_n)$, so they are not suited to the study of this particular problem.
\end{rmq}

\subsection{The stable \texorpdfstring{$p$}{p}-restricted Andreadakis problem}

\subsubsection{Vanishing of the trace map}

In the $p$-restricted context, Proposition \ref{Tr_o_tau} is replaced by:

\begin{prop}\label{Tr_o_tau-p}
Let $k \geq 2$, and let $J \in GL_m(I^k_{\F_p} F_n)$. Then:
\[\tr \left( J - \Id{} \right) \in [TV,TV]_k + (TV)^p \subset V^{\otimes k} \cong I^k/I^{k+1},\]
Where $V = F_n^{ab} \otimes \F_p \cong \F_p^n$.
\end{prop}

The proof is exactly the same as the proof of proposition $\ref{Tr_o_tau}$, over $\F_p$ instead of $\Z$, Proposition \ref{criterion_for_brackets} being replaced by:

\begin{prop}\label{criterion_for_brackets-p}
Let $f(X_1, ..., X_n) \in V^{\otimes k}$. Let $\mathcal C \subset M_k(\F_p)$ be the sub-$\Z$-module generated by the $\textbf{e}_{i,i+1}$. Suppose: 
\[\forall C_i \in \mathcal C,\ \tr(f(C_1, ..., C_n)) = 0.\]
Then $f \in [TV,TV]_k + (TV)^p$.
\end{prop}

\begin{proof} 
We say that two elements $u$ and $u'$ of $V^{\otimes k}$ are \emph{cyclically equivalent}, and we write $u \sim u'$ if they are conjugate under the action of $\Z/p$. We want to show that $f$ is cyclically equivalent to a $p$-th power. Let us decompose $f$, up to cyclical equivalence, as a sum of pairwise non-cyclically equivalent monomials: $f \sim \sum \mu_g g$, where each $g$ is of the form $g = X_{i_1} \cdots X_{i_k}$. If $g$ is such that the $i_\alpha$ are pairwise distinct, evaluate each $X_{i_\alpha}$ as $C_{i_\alpha} = \textbf{e}_{\alpha, \alpha + 1}$, and all $X_i$ not appearing in $g$ as $C_i = 0$. Then $\mu_g = \tr(f(C_1, ..., C_n)) = 0$. As a consequence, no such $g$ can appear in our decomposition of $f$. 

Let $\lambda$ be the algebra morphism from $TV$ to $T(V \otimes V)$ sending $X_i$ to $\sum_j X_{ij}$ (where $X_{ij} = X_i \otimes X_j$). Take a monomial $g = X_{i_1} \cdots X_{i_k}$ as above. Its image is $\lambda(g) = \sum_{\textbf j} X_{i_1 j_1} \cdots X_{i_k j_k}$, the sum being taken over every $\textbf j = (j_1, ..., j_k) \in \{1, ..., k\}^k$. Let $r$ be the number of monomials cyclically equivalent to $h = X_{i_1 1} \cdots X_{i_k k}$ in this sum. Then $r$ is exactly the number of elements of $\Z/k$ stabilizing $g$. It is a multiple of $p$ if and only if $g$ is a $p$-th power. If we decompose $\lambda(f)$ up to cyclic equivalence, as we did earlier for $f$, the only occurrences of $h$ must come from $\lambda(g)$, hence the coefficient of $h$ must be $r \mu_g$. Note that $\lambda(f)$ satisfies the same hypothesis as $f$, because $\mathcal C$ is stable under addition. Since the $X_{i_\alpha \alpha}$ are pairwise distinct, we can apply the above argument and find that $r \mu_g = 0$. Thus, $\mu_g = 0$ or $g$ is a $p$-th power. Whence the result.
\end{proof}

\begin{rmq}
Any bracket and any $p$-th power satisfies the condition of Proposition \ref{criterion_for_brackets-p}. For brackets, it follows from the fact that the trace of a bracket is itself a sum of brackets. For $p$-th power, remark that if $M = \sum m_i \textbf{e}_{i, i+1} \in M_n(R)$, where $R$ is an associative ring of caracteristic $p$, then $M^p = (\prod m_i) \cdot \Id{p}$, and $\tr(\Id{p}) = p\cdot 1 = 0$ in $\k$. If $k = pl$ and $f = (X_{i_1} \cdots X_{i_l})^p$, apply this to $M = C_{i_1} \cdots C_{i_l}$ (where the $C_{i_\alpha}$ are in $\mathcal C$), seen as a $p\times p$-matrix with coefficients in $M_l(\k)$. 
\end{rmq}

Because of Proposition \ref{Tr_o_tau-p}, in characteristic $p$, we will consider the trace map as taking values in  $C_*^{[p]}V = TV/([TV,TV] + (TV)^p)$, which is the quotient of the cyclic power $C_*V$ by $p$-th powers.
The conclusion of Proposition \ref{Tr_o_tau-p} then becomes: $\tr \left( J - \Id{} \right) = 0 \in C_*^{[p]}V.$

\subsubsection{Linear algebra}

Consider the Johnson morphism 
\[(\tau^{[p]})': \Lie^{[p]}\left(IA_n^{[p]} \right) \longrightarrow \Der^{[p]}_*\left(\Lie^{[p]}(F_n) \right) \cong V^* \otimes \mathfrak L^{[p]}V,\]
the last isomorphism being obtained as in Example \ref{Johnson_Fn}, using Example \ref{Lp(Fn)_free} instead of Example \ref{L(Fn)_free}, and replacing derivations by $p$-restricted ones.
\emph{When $p \neq 2$}, the morphism $(\tau^{[p]}_1)'$ is surjective. Indeed, the free $\F_p$-Lie algebra $\mathfrak LV$ is a sub-algebra of the free $p$-restricted algebra $\mathfrak L^{[p]}V$, and this inclusion is an isomorphism in degrees prime to $p$, in particular in degree $2$; thus we can lift the generators of $V^* \otimes \Lambda^2V$ by the generators of $IA_n$ used in the proof of Proposition \ref{IAn^ab}. Moreover, $\Lie^{[p]} \left(IA_n \right)$ is generated in degree $1$ as a $p$-restricted Lie algebra (cf. \ref{engdeg1-p}). As a consequence, the image of $\tau'$ is exactly the $p$-restricted Lie algebra generated in degree $1$ inside $\Der^{[p]}_*(\mathfrak L^{[p]}V )$.

\medskip

The reader can easily check that the obvious $p$-restricted version of Proposition \ref{algebraic_description_of_tr} does hold: the trace map obtained from free differential calculus can be seen as the composite of the Johnson morphism $\tau: \Lie_k(\mathcal A_*^{[p]}(F_n)) \rightarrow \Der_k^{[p]}(\mathfrak L^{[p]}V) \cong V^* \otimes \mathfrak L^{[p]}_{k + 1}V$ with:
\[\tr_M: V^* \otimes \mathfrak L^{[p]}_{k + 1}V \overset{\iota}{\longrightarrow} 
        V^* \otimes V^{\otimes k + 1}    \overset{\Phi}{\longrightarrow}
        V^{\otimes k}                    \overset{\pi}{\longrightarrow}
        C_k^{[p]}V,  \]
where $\iota$ and $\pi$ again denote the canonical maps.

\begin{nota}
Let $\mathfrak I^{[p]}$ denote the image of $(\tau^{[p]})'$, the $p$-restricted Lie algebra generated in degree $1$ inside $\Der(\mathfrak LV)$.
\end{nota}

The following proposition can be seen as a direct consequence of Proposition \ref{Tr_o_tau-p}.
\begin{prop}\label{Tr(I)-p}
For every $k \geq 2$, $\tr_M(\mathfrak I_k^{[p]}) = \{0\}.$
\end{prop}

Consider be the subspace of $p$-restricted derivations stabilizing the free $\F_p$-Lie algebra $\mathcal LV \subset \mathcal L^{[p]}V$. It is a $p$-restricted Lie sub-algebra of $\Der^{[p]}(\mathcal L^{[p]}V)$. Since each derivation of $\mathfrak LV$ extends to a unique $p$-restricted derivation of $\mathcal L^{[p]}V$, this sub-algebra is isomorphic to $\Der(\mathfrak LV)$. Under the identification with the graded module $V^* \otimes \mathfrak L^{[p]}V$, it corresponds exactly to $V^* \otimes \mathfrak LV$. As a consequence, if $p \neq 2$, the degree one part is the same. Hence:
\[\mathfrak I^{[p]} \subseteq \Der^{[p]}_{\mathcal LV}(\mathcal L^{[p]}V).\]
This implies that there is \textbf{no stable surjectivity here}: we can easily easily produce examples of automorphisms whose associated derivation does not preserve $\mathcal LV$. For instance, take any world in $\Gamma_k(F_n)$ not containing any occurrence of $x_1$. Then the automorphism $\varphi$ defined by $x_1 \mapsto w^p x_1$ and $x_i \mapsto x_i$ when $i \neq 1$ is obviously in $\mathcal A^{[p]}_{pk-1}$, but $\tau(\varphi) = X_1^* \otimes (\overline{w - 1})^p$ sends $X_1$ outside of $\mathfrak LV$.

\subsubsection{Stable cokernel of \texorpdfstring{$i_*$}{i*}}\label{coker-p}

We close the present paper with a quantification of the lack of stable surjectivity in the $p$-restricted case.

Let $k \geq 2$. Like in paragraph \ref{coker}, we get a commutative diagram with exact rows:
\[\begin{tikzcd}
\bar{\mathfrak I}_k              \ar[r, hook] \ar[d, dashed, "\phi"]
& V^* \otimes \mathfrak L_{k+1}V \ar[r, two heads]     \ar[d, "\Phi"]
& X_k                            \ar[d, dashed , "\overline{\Phi}"] \\
{[TV, TV]_k}      \ar[r, hook]   
& V^{\otimes k}   \ar[r, two heads, "\pi"]               
& C_k V.
\end{tikzcd}\]
Here, $V$ denotes $F_n^{ab} \otimes \F_p$, and $\bar{\mathfrak I}_* = \mathfrak I_* \otimes \F_p$ is the sub-Lie algebra generated in degree one inside $\Der(\mathfrak LV)$. The space $X_k$ is just the quotient of $V^* \otimes \mathfrak L_{k+1}V$ by $\bar{\mathfrak I}_k$.

We will be interested in a slightly different diagram, though:
\[\begin{tikzcd}
\mathfrak I_k^{[p]}              \ar[r, hook] \ar[d, dashed, "\phi"]
& V^* \otimes \mathfrak L_{k+1}V \ar[r, two heads]     \ar[d, "\Phi"]
& X'_k                           \ar[d, dashed , "\overline{\Phi}"] \\
{[TV, TV]_k + (TV)^p}      \ar[r, hook]   
& V^{\otimes k}            \ar[r, two heads, "\pi"]               
& C_k^{[p]} V.
\end{tikzcd}\]
We can still apply the calculations from \cite{Satoh2} to show that if $n \geq k+2$, then $\Phi$ is surjective and $\ker \Phi \subseteq \bar{\mathfrak I}_k \subseteq \mathfrak I_k^{[p]}$. Only, now $\phi$ could have a cokernel. From \cite{Satoh2}, we only get that brackets are in its image, so this cokernel can only come from $p$-th powers. In particular, it is concentrated in degrees divisible by $p$. We denote it by $K$. The same application of the snake lemma as in the proof of Proposition \ref{prop-coker} gives that $K$ is also the kernel of $\overline\Phi$, and that $\overline\Phi$ is surjective.

Now consider the diagram:
\[\begin{tikzcd}
\mathfrak I_k^{[p]}              \ar[r, hook] \ar[d, hook, "\iota"]
& V^* \otimes \mathfrak L_{k+1}V \ar[r, two heads]     \ar[d, hook]
& X'_k                           \ar[d, two heads , "\overline{\Phi}"] \\
\ker(\tr_M)_k      \ar[r, hook]   
&V^* \otimes \mathfrak L_{k+1}^{[p]}V            \ar[r, two heads, "\tr_M"]               
& C_k^{[p]} V.
\end{tikzcd}\]
Denote by $L$ the cokernel of the middle inclusion, then $L = V^* \otimes (\mathfrak L_{k+1}^{[p]}V/\mathfrak L_{k+1}V)$ is concentrated in degrees $ k = pl - 1$ (with $l \geq 1$). The snake lemma gives a short exact sequence:
\[\begin{tikzcd}
0 \ar[r] &K \ar[r] &\coker(\iota) \ar[r] &L \ar[r] &0.
\end{tikzcd}\]
Since the trace map $\tr_M \circ \tau$ vanishes, we have a commutative diagram:
\[\begin{tikzcd}
\Lie(IA_n^{[p]}) \ar[r, "i_*"] \ar[rd, swap, "\tau'"] 
&\Lie(\mathcal A_*^{[p]}) \ar[d, hook, "\tau"] \\
&\ker(\tr_M).
\end{tikzcd}\]
This implies that the cokernel of $i_*$ injects into the cokernel of $\iota$. Thus, we have proved:
\begin{prop}\label{lack_of_stable_surj-p}
Fix $n$ an integer, and consider only degrees $k \leq n-2$. The cokernel of the canonical morphism 
\[i_*: \Lie(IA_n^{[p]}) \rightarrow \Lie(\mathcal A_*^{[p]}(F_n))\]
is concentrated in degrees $k = pl-1$ and $k = pl$ (for $l \geq 1$). 
If $k = pl-1$, then $\coker((i_*)_k)$ injects into $V^* \otimes (\mathfrak L_{k+1}^{[p]}V/\mathfrak L_{k+1}V)$. If $k = pl$, it is a sub-quotient of $V^{\otimes l}$.
\end{prop}

\begin{rmq}
The tensor power $V^{\otimes l}$ appearing in the proposition is in fact the \emph{Frobenius twist} of $V^{\otimes l}$. This has no consequence here, as the Frobenius map is trivial on $\F_p$, but it should be kept in mind for any functorial study of this situation.
\end{rmq}

\bibliographystyle{alpha}
\bibliography{Ref_Surj_stable}

\end{document}